\documentclass[oneside]{amsart}
 \usepackage{stmaryrd}
\usepackage{hyperref}
\usepackage{cleveref}
\usepackage{soul}

\usepackage{titletoc}
\usepackage{wrapfig}

\usepackage{paralist}
\usepackage{url}
\usepackage{enumerate}
\usepackage{rotating}
\usepackage{color}
 \usepackage{tikz-cd}
\newcommand{\xdownarrow}[1]{
  \ensuremath{\begin{turn}{90}{\tiny${#1 }$}\end{turn}
    \left\downarrow\vbox to 0.4cm{}\right.\kern-\nulldelimiterspace
  }
}
\usepackage{mathtools}
\newcommand{\xuparrow}[1]{
  \ensuremath{
    \begin{turn}{90}{\tiny ${#1}$}\end{turn} \left\uparrow\vbox to 0.4cm{}\right. \kern-\nulldelimiterspace
  } 
}
\definecolor{asparagus}{rgb}{0.53, 0.66, 0.42}
\usepackage[all,cmtip]{xy}
\usepackage{ mathrsfs }

\DeclareMathOperator{\Pro}{Pro}

\newcommand{\moduli}{\mathrm{Moduli}}
\newcommand{\liep}{\mathrm{Lie}^{\pi}_{F}}

\usepackage{filecontents} 
\newcommand{\art}{\mathrm{art}}
\newcommand{\ZZ}{\mathbb{Z}}

\newcommand{\EE}{\mathbb{E}}

\newcommand{\FF}{\mathbb{F}}

\newtheorem*{theorem*}{Theorem}

\DeclareMathOperator{\Def}{Def}  
\DeclareMathOperator{\Fr}{Fr}   
\DeclareMathOperator{\op}{op}   
\DeclareMathOperator{\coAlg}{coAlg}   
  
\DeclareMathOperator{\wafp}{wafp}  
\DeclareMathOperator{\colim}{colim}  
\DeclareMathOperator{\Free}{Free}

\usepackage{upgreek}
\newcommand{\mytimes}[1]{\mathbin{\operatorname*{\times}_{#1}^{}}}

\DeclareMathOperator{\ad}{ad}
\DeclareMathOperator{\Poly}{Poly}
\DeclareMathOperator{\Fun}{Fun}
\DeclareMathOperator{\Perf}{Perf}
\DeclareMathOperator{\Forget}{Forget}

\DeclareMathOperator{\trdeg}{trdeg}

\usepackage[small,nohug,heads=vee]{diagrams}

 \xyoption{rotate}

\DeclareMathOperator{\Lie}{Lie}

\DeclareMathOperator{\sqz}{sqz} 
\DeclareMathOperator{\Spf}{Spf} 
\DeclareMathOperator{\fields}{Fields} 
 
\DeclareMathOperator{\Spec}{Spec} 
\DeclareMathOperator{\Ind}{Ind} 
\DeclareMathOperator{\cofib}{cofib} 
 
\DeclareMathOperator{\End}{End} 
\DeclareMathOperator{\triv}{triv} 
\DeclareMathOperator{\Pperf}{Perf} 
\DeclareMathOperator{\APperf}{APerf} 
 
\DeclareMathOperator{\Map}{Map}

\DeclareMathOperator{\aug}{aug}

\DeclareMathOperator{\fib}{fib} 
 
\DeclareMathOperator{\Lan}{Lan} 
\DeclareMathOperator{\Gal}{Gal}

\DeclareMathOperator{\Tot}{Tot}

\DeclareMathOperator{\cN}{cN} 
 
\DeclareMathOperator{\ft}{ft} 

\newcommand{\myotimes}[1]{\mathbin{\operatorname*{\otimes}_{#1}}}
\DeclareMathOperator{\Alg}{Alg}
\DeclareMathOperator{\aaalg}{alg}

\DeclareMathOperator{\Der}{Der}

\DeclareMathOperator{\LieAlgd}{LieAlgd}

\DeclareMathOperator{\id}{id}  
\DeclareMathOperator{\Sp}{Sp}

\usepackage{verbatim}
\usepackage{geometry}
\usepackage{amsfonts}
\usepackage{amssymb}
\usepackage{amsmath}
\usepackage{amsthm}

\usepackage{xy}
\input xy

\xyoption{all}

\renewcommand{\mod}{\mathrm{Mod}}

\theoremstyle{definition}
\newtheorem{definition}{Definition}[section]

\usepackage{upgreek}
\newtheorem{construction}[definition]{Construction}
\newtheorem{example}[definition]{Example}
\newtheorem*{example*}{Example}

\newtheorem{notation}[definition]{Notation}
\newtheorem{remark}[definition]{Remark}
\theoremstyle{theorem}
\newtheorem{proposition}[definition]{Proposition}
\newtheorem{lemma}[definition]{Lemma}
\newtheorem{corollary}[definition]{Corollary}

\newtheorem{theorem}[definition]{Theorem}

\newcommand{\SCR}{\mathrm{SCR}}

\newcommand{\Mod}{\mathrm{Mod}}

\begin{document}

\title{Purely Inseparable Galois theory I: \\  The Fundamental Theorem}
 
\author[Lukas Brantner]{Lukas Brantner}
\address{Oxford University,   Universit\'{e} Paris--Saclay (CNRS)} 
\email{lukas.brantner@maths.ox.ac.uk, lukas.brantner@universite-paris-saclay.fr}

\author[Joe Waldron]{Joe Waldron}
\address{Michigan State University}
\email{waldro51@msu.edu}

\begin{abstract} 
We construct  a Galois correspondence for finite purely inseparable field \mbox{extensions $F/K$}, generalising a classical result of Jacobson for extensions of exponent one  \mbox{(where $x^p \in K$ for  all $x$).}
\end{abstract} 
\maketitle
 
\tableofcontents   

\section{Introduction\vspace{-1pt}}
A finite extension of fields $F/K$ of characteristic $p$ is  \textit{purely inseparable} if for each $x\in F$, there is some $n$ such that $x^{p^n}\in K$.
Any  {finite} field extension $F/k$ can be broken down into a separable \mbox{part $K/k$  and} a purely inseparable part $F/K$. While the intermediate extensions of $K/k$  can be understood using classical Galois theory \cite{galois1846lettre}, the situation is more subtle  {for the purely inseparable extension $F/K$.}  {For instance if $F/K$ is purely inseparable, the classical Galois group satisfies $\Gal(F/K)=0$. } \vspace{2pt}

If $F/K$ has exponent one, that is, if  $x^p$ belongs to $K$  for all  $x\in F$, then Jacobson \cite{MR11079} classified intermediate extensions of $F/K$ in terms of  the restricted Lie algebra $\Der_K(F)$ \mbox{of derivations.}
Later, Sweedler \cite{MR223343},  Gerstenhaber--Zaromp \cite{MR266904}, and Chase \cite{CIS-333302}    extended \mbox{Jacobson's} correspondence to modular subextensions of  modular extensions,  by using the  higher derivations of Hasse and Schmidt  \cite{MR1581557}. \vspace{2pt}

In \Cref{fundamental_theorem} of   this article, we  establish a Galois correspondence for general finite purely \mbox{inseparable}  field extensions by using the methods of derived algebraic geometry.\vspace{-1pt}

\subsection{Background\vspace{-2pt}}
We begin by introducing the main objects of study:
\begin{definition}[Purely inseparable extensions] Let $F/K$ be a field  extension in \mbox{characteristic $p>0$.}
\begin{enumerate}
\item The extension is said to be \textit{purely inseparable} if for any $x\in F$,   {there is $i$ such that} $x^{p^i}$ \mbox{belongs to $K$.}
\item The extension has \textit{exponent $n$} if $x^{p^n}\in K$ for \textit{all} $x\in F$, and $n$ is minimal with this property.
\end{enumerate}
\end{definition}

Given a finite purely inseparable field extension $F/K$, we can consider the $F$-vector space\vspace{-1pt}
$ \Der_K(F)$
 of $K$-linear derivations of $F$, that is, of  $K$-linear maps $D: F \rightarrow F$ satisfying  $$D(xy) = D(x)y + x D(y) \mathrm{\ \ \ \ for\ all\ } x,y\in F$$ for all $x,y\in F$. Observe that  this forces $D(a)=0$ for all $a\in K$.\vspace{3pt} 

The $F$-vector space $   \Der_K(F)$ is equipped with a $K$-bilinear   Lie bracket 
$$ [-,-]: \Der_K(F) \times \Der_K(F) \rightarrow \Der_K(F)$$
$$  \ \ \  (D_1 , D_2) \mapsto  [D_1,D_2] := D_1 \circ D_2 - D_2 \circ D_1 $$
and a self map, called the restriction, which is given by
$$(-)^{[p]} :  \Der_K(F) \rightarrow  \Der_K(F)$$
$$  \ \ \   \ \ \  \ \ D \mapsto D^{[p]}:= \underbrace{D \circ \ldots \circ D}_{p}.$$
Denoting by 
 $\ad(x)$ the adjoint representation $\Der_K(F) \rightarrow \Der_K(F),  y \mapsto [x,y]$,  we have
 for all  derivations $ D_1, D_2 \in \Der_K(F)$ \mbox{and all scalars $\lambda \in K$:}
 \begin{enumerate} 
\item $(\lambda D_1)^{[p]} = \lambda^p D_1^{[p]}$;\vspace{6pt}
\item $\ad(D_1^{[p]}) = \ad(D_1)^{p}$;\vspace{-3pt}

\item  $ \displaystyle (D_1+D_2)^{[p]}  \hspace{-1pt} \hspace{-1pt}=  \hspace{-1pt}\hspace{-1pt} D_1^{[p]} +   \sum_{i=1}^{p-1} s_i(D_1, D_2)+ D_2^{[p]}$  

 {Here $is_i(D_1,D_2)$ is the coefficient {of $t^{i-1}$ \hspace{-3pt} in   $\ad(tD_1 \hspace{-1pt}+ \hspace{-1pt} D_2)^{p-1}(D_1)$, and so is a linear combination of Lie brackets of $D_1$ and $D_2$.}}
\end{enumerate}
A Lie algebra with a self-map  {$(-)^{[p]}$} satisfying these three conditions on 
 is called a \textit{restricted Lie algebra}.  {For more details on restricted Lie algebras see \cite{Jacobson_restricted}.}
 
 \vspace{4pt}

More generally, if $K \subset E \subset F$ is an intermediate field, we can consider  the $F$-linear injection
\begin{equation} \label{anchor1} \rho:  \Der_E(F) \hookrightarrow \Der_K(F).\end{equation}
Again, $\Der_E(F)$ carries a Lie bracket and a restriction   satisfying the above axioms. The \mbox{\textit{anchor map}   $\rho$} is compatible with   both bracket and restriction,  and  satisfies the following equation for all $\phi \in F$:
$$[X,\phi Y] = \phi [X,Y] + \rho(X)(\phi) \cdot Y$$
One might therefore think of   $(\Der_E(F)  \hookrightarrow \Der_K(F))$ as an \textit{$F/K$-restricted Lie algebroid}
whose  anchor \vspace{3pt}map happens to be  injective. \label{FKrestricted}

\hspace{-5pt} If $x^p\in K$ for \vspace{-3pt} all elements $x\in F$, this construction induces   \mbox{the following correspondence  \cite{MR11079}:}
\begin{theorem}[Jacobson] \label{Jacob}
Let $F/K$ be a finite purely inseparable field extension of \mbox{exponent one.}
There is an inclusion-reversing bijection  between  $F/K$-restricted Lie algebroids with injective anchor  map $(\mathfrak{g} \hookrightarrow \Der_K(F))$
and  intermediate field extensions $K \subset E \subset F$.

The Lie algebroid   corresponding to an intermediate field $E$ is \mbox{given by $(\Der_E(F) \hookrightarrow \Der_K(F)).$}

The field corresponding to an algebroid  $(\mathfrak{g} \hookrightarrow \Der_K(F))$ is  
\mbox{$ \{ x\in F \  | \  D(x) = 0 \mbox{ for  } D\in \mathfrak{g}\}\subset F.$}  
\end{theorem}

Jacobson's   theory extends from the generic fibre to give a geometric \mbox{correspondence \cite{rudakov_shafarevich, ekedahl}:}
\begin{theorem}[]
Let $X$ be a normal variety over an algebraically closed field \mbox{$k$ of  {characteristic $p>0$.}}  Then there is an inclusion-reversing bijection between saturated subsheaves of  $\mathcal{T}_X$ which are closed under   Lie bracket and $p$-powers (called \textit{foliations}), and morphisms  
$$X \rightarrow Y $$
with $Y$ normal factoring the absolute Frobenius $\Fr: X \rightarrow X$. 
\end{theorem}
This extension of Jacobson's Galois theory to varieties has led to many further geometric applications, for example  \cite{shepherd-barron_geography_1991}, \cite{shen_foliations_2009}, \cite{di_cerbo_effective_2015}, \cite{langer2015generic}, \vspace{5pt} \cite{patakfalvi_singularities}, and \cite{ji_waldron}.  
 
A major drawback of Jacobson's theory is that it can only be applied to exponent one extensions. Indeed, if $F/K$ has exponent larger than one, the assignment $E \ \mapsto\  (\Der_E(F) \hookrightarrow \Der_K(F))$ can no longer distinguish between all intermediate fields,   and  \vspace{-2pt}Jacobson's \mbox{correspondence breaks down:}

\begin{example}
 {Let $F$ be a field of characteristic $p$. } 
 Then $\Der_{F^{p^n}}(F)=\Der_{F^p}(F)=\Der_{\mathbb{F}_p}(F)$, as for any derivation  $D\in\Der_{\mathbb{F}_p}(F)$, the Leibnitz rule implies that  
for all $x\in F$, we have 
$$D(x^p)=px^{p-1}D(x)=0.$$ Hence  derivations cannot tell  the extensions $F/F^p$ and $F/F^{p^n}$ apart.   {Note that if $X$ is a positive dimensional variety and $F=K(X)$ is its function field, $F/F^{p^n}$ is a purely inseparable extension of exponent exactly $n$, and so $F^{p^n}$ are distinct fields for distinct $n$. On the other hand, if $F$ is a perfect field, a similar argument shows that $\Der_{\mathbb{F}_p}(F)=0$.}
\end{example}

Jacobson's theory has  been extended to modular subextensions, i.e.\ subextensions of the form $$F = E[X_1]/(X_1^{p^{n_1}}-a_1) \otimes \ldots \otimes  E[X_k]/(X_k^{p^{n_k}}-a_k).$$
The key property of these extensions was established by Sweedler \cite{MR223343}, \mbox{who proved that} $F/E$ is modular  if and only if $E$ is the fixed field of  a set of Hasse--Schmidt \mbox{derivations $(D_0, \ldots, D_m)$ of $F$.} This fundamental fact was later used to study the Galois theory of modular extensions by  Chase \cite{CIS-333302},   \vspace{2pt} Heerema--Deveney \cite{heerema1974galois}, \mbox{Gerstenhaber--Zaromp \cite{MR266904}, and others.}

However, not all purely inseparable field extensions  are modular -- we recall the following simple example by Sweedler (cf.\ \cite[Example 1.1]{MR223343}):
\begin{example} 
For $K=\FF_p(x^p,y^p,z^{p^2})$ and $F= K[xz+y, z]$, the extension $F/K$ is not modular.  \vspace{2pt}
\end{example}

\subsection{Techniques}
In this article, we  establish a Galois correspondence for arbitrary finite purely inseparable field extensions.
To retain the  information lost by Jacobson's functor $$(K \subset E \subset F) \ \ \ \  \mapsto \ \ \ \  (\Der_E(F) \hookrightarrow \Der_K(F)),$$
we will construct a  \vspace{-2pt} refinement using the cotangent complex formalism in \vspace{3pt} derived algebraic geometry.

Recall that for any map of rings $A \rightarrow B$ , the \textit{cotangent complex} $L_{B/A}$ is the complex of $B$-modules 
 obtained by first resolving $B$ by  a free simplicial $A$-algebra $P_\bullet$,  then applying the K\"{a}hler differentials  functor $\Omega^1_{-/A}$ in each simplicial degree, and finally applying the Dold-Kan correspondence.
\mbox{The zeroth homology} group $\pi_0(L_{B/A} )$   is given by  $\Omega^1_{B/A}$. Note that here and everywhere else in this paper, we write $\pi_i(M)$ for the $i^{th}$ homology group of a chain complex, which is equal to the $i^{th}$ homotopy group of the associated module spectrum.

If $F/K$ is a finite field extension,  pick   $x_1,\ldots, x_n\in F$  such that  the following map  is surjective with kernel $I$:
$$K[X_1,\dots X_n] \xrightarrow{X_i \mapsto x_i} F.$$   {As we will see in \Cref{prop:cotangent_calculation},} the cotangent complex \vspace{-2pt} $L_{F/K}$ is then concentrated in two  \mbox{degrees, and given by}
$$ L_{F/K} = \left( \ldots  \rightarrow 0 \rightarrow  I/I^2 \xrightarrow{[i] \mapsto di \otimes 1} \Omega^1_{K[X_1,\ldots, X_n]/K} \otimes_{K[X_1,\ldots, X_n]}  F\right).  \vspace{-2pt}$$ 

 {If $F/K$ is separable,  then $L_{F/K}$ vanishes,  analogously to the vanishing of the usual Galois group for purely inseparable extensions.   Hence our Galois theory is perpendicular to the classical Galois theory of separable extensions, and from now on,  we shall assume that $F/K$ is purely inseparable.}

To refine Jacobson's functor \eqref{anchor1}, we will consider the assignment
\begin{equation} \label{anchor2} (K \subset E \subset F) \ \ \ \  \mapsto \ \ \ \  (L_{F/E}^\vee[1] \rightarrow L_{F/K}^\vee[1]),\end{equation}
where $(-)^\vee$ denotes the $F$-linear dual of a chain complex  \vspace{2pt} and $[1]$ is a homological shift by $+1$.

It is of course not enough to consider this\vspace{-1pt}  assignment \eqref{anchor2} merely as a functor to $(\Mod_F)_{/L_{F/K}^\vee[1]},$ the \mbox{$\infty$-category} of  $F$-linear chain complexes with a map to  $L_{F/K}^\vee[1]$, just like Jacobson's functor  did not just record $(\Der_E(F) \rightarrow \Der_K(F))$ as a map of $F$-vector spaces.\vspace{3pt}

Instead, we  equip  $(L_{F/E}^\vee[1] \hookrightarrow L_{F/K}^\vee[1])$ with a  kind of  derived restricted  Lie algebroid structure.  

To this end, we elaborate\vspace{0.4pt}  on the theory of \textit{partition Lie algebras}, which was introduced in  \cite{brantner2019deformation}  to study infinitesimal  \vspace{0.4pt}deformations in characteristic $p$.  More precisely,    for any field $F$, there is an equivalence\vspace{-0.2pt}  between $\moduli_F$, the $\infty$-category of formal moduli problems, and $\mathrm{Alg}_{\liep}$, the \mbox{$\infty$-category} of partition Lie algebras, \mbox{cf.\ \cite[Theorem 1.11]{brantner2019deformation}.} In [op.cit.], an additional subscript $(-)_{\Delta}$  highlights that we work  in the setting of simplicial commutative rings \mbox{(rather than $\EE_\infty$-rings).}

\begin{remark} In characteristic zero, partition Lie algebras are simply (shifted) differential graded Lie algebras, and \cite[Theorem 1.11]{brantner2019deformation} recovers   the Lurie--Pridham theorem, cf.\ \cite{lurie2011derivedX}\cite{pridham2010unifying}.   \end{remark}
Partition Lie algebras are defined by a monad $\liep$  on $\Mod_F$ satisfying the \mbox{following properties:} 
\begin{construction}[Partition Lie algebras]\label{pliecon}\

\begin{enumerate}
\item  $\liep$ commutes with sifted colimits, i.e.\ with 
filtered colimits and
\mbox{geometric realisations.}
\item If $V\in \Mod_F$ is   coconnective and  $\pi_i(V)$ is finite-dimensional for all $i$, then $\liep(V)$ is given by $L_{F/F\oplus V^\vee}^\vee[1]$, where 
$F \oplus
V^{\vee}$ denotes  the \vspace{2pt} trivial square-zero  \mbox{extension of $F$ by $V^\vee$.}\\
Note that any  $W\in \Mod_F$  is a sifted colimit of chain complexes $V$ of the above form. \vspace{2pt} 
\item  {If $V^\bullet$  is  a cosimplicial $F$-vector space with totalisation $V = \Tot(V^\bullet)$ , then} $$ \liep(V) \simeq \bigoplus_{n} \Tot \left(\widetilde{C}^\bullet(\Sigma |\Pi_n|^\diamond,F) \otimes (V^\bullet)^{\otimes n}\right)^{\Sigma_n}.  $$ 
 {Here $\widetilde{C}^\bullet(\Sigma |\Pi_n|^\diamond,F)$ are  the  $F$-valued cosimplices of the $n^{th}$ partition complex (cf.\ e.g.\ }\cite{arone2018action}),  {the functor $(-)^{\Sigma_n}$ takes strict fixed points,  and the tensor product is computed in cosimplicial $F$-modules.}
 
\item  {The  tangent fibre $\cot_F^\vee(R)=L_{F/R}^\vee[1]$
of any augmented simplicial commutative $F$-algebra $R\in \SCR_F^{\aug}$
carries  a canonical
$\liep$-algebra structure.   
Note that there is a natural equivalence $ \cot_F^\vee(R) \simeq  (F \otimes_R L_{R/F})^\vee$ as the 
the composite $F\to R\to F$
induces  a 
fibre sequence $F \otimes_R L_{R/F} \rightarrow L_{F/F} \rightarrow L_{F/R}$, and $L_{F/F}\simeq 0$ .}
 
\end{enumerate} 
\end{construction}
\begin{remark}
Under the natural grading conventions we  adopt, the homotopy groups $\pi_\ast(\mathfrak{g})$ of any partition Lie algebra  form a \textit{shifted}   {Lie algebra,   {which means that} the} bracket  $$[-,-]:  \pi_i(\mathfrak{g}) \times \pi_j(\mathfrak{g}) \rightarrow \pi_{i+j-1}(\mathfrak{g})$$  preserves $\pi_1(\mathfrak{g})$ and satisfies
$[x,y] = (-1)^{|x||y|}[y,x]$, as well as  the usual graded   Jacobi identity.
 {Shifted   Lie brackets are familiar from
 homotopy theory, where they appear as Whitehead products.} 
\end{remark}

\begin{example}\label{firstex}
For any field extension $K\subset F$, we can construct two different partition Lie algebras:
\begin{enumerate}
\item The tangent fibre of the representable $F$-formal moduli problem $\Spf(F \otimes_K F)$  {(given by the functor of points, see \Cref{ex:functor_points})}, which encodes deformations of the diagonal, is a partition Lie algebra over $F$ with underlying chain complex  $$L_{F/K}^\vee \simeq \cot(F \otimes_K F)^\vee.$$ 
\item Infinitesimal deformations of 
the $K$-scheme $\Spec(F)$ are encoded by a \mbox{Kodaira--Spencer  formal} moduli problem; the corresponding 
partition Lie algebra over $K$ has underlying\vspace{2pt}  \mbox{chain complex}  $$L_{F/K}^\vee[1].\vspace{2pt} $$
\end{enumerate}
\end{example}\ 

\subsection{Statement of Results} Fix\vspace{2pt}  a finite purely inseparable field extension $F/K$.\vspace{2pt}

We will \vspace{-1pt} construct a monad\vspace{1pt} $\LieAlgd^\pi_{F/K}$ acting on the $\infty$-category $(\Mod_F)_{/L_{F/K}^\vee[1]}$ of arrows \mbox{$M\rightarrow L_{F/K}^\vee[1]$}.
We \vspace{-1pt}   refer to $\LieAlgd^\pi_{F/K}$-algebras as \textit{$F/K$-partition Lie algebroids}, and \vspace{-1pt}  
denote the resulting $\infty$-category
by
$\Alg_{\LieAlgd^\pi_{F/K} }$\hspace{-2pt}. Given $(\mathfrak{g} \hspace{-1pt} \xrightarrow{\rho}\hspace{-1pt} L_{F/K}^\vee[1])  \hspace{-1pt} \in \hspace{-1pt} \Alg_{\LieAlgd^\pi_{F/K} }$\hspace{-3pt} , we \vspace{4pt} \mbox{call $\rho$ the \textit{anchor map}.}

 $F/K$-partition Lie algebroids are   derived generalisations of  the classical $F/K$-restricted Lie algebroids on p.\ \pageref{FKrestricted}.  \Cref{maincons} below will 
list the key  properties of  $\LieAlgd^\pi_{F/K}$ -- for now, let us simply record that for 
any simplicial commutative $K$-algebra $B$ with a map to $F$,  \mbox{the basic
arrow} $$\left(L_{F/B}^\vee[1] \rightarrow L_{F/K}^\vee[1]\right)$$
can be equipped with a canonical $F/K$-partition Lie algebroid structure $\mathfrak{D}(B)$. 

\begin{definition}[Galois partition Lie algebroids]
Given an intermediate field $E$ of the
 finite purely inseparable field extension $F/K$,  the  \textit{Galois partition Lie algebroid} is given by
$\mathfrak{gal}_{F/K}(E) := \mathfrak{D}(E);$ its
underlying object is given by the arrow 
of chain complexes   $( L_{F/E}^\vee[1] \rightarrow L_{F/K}^\vee[1])$.
\end{definition}

Our main result is the following  generalisation of Jacobson's correspondence to \mbox{arbitrary exponents:}

\begin{theorem}[Fundamental theorem of purely inseparable Galois theory]\label{fundamental_theorem} \ \label{mainintro} \\
Let $F/K$ be a finite purely  {inseparable field extension.}
Then there is a contravariant equivalence between the poset of  intermediate field extensions\vspace{-2pt} $$K \subset E \subset F\vspace{-2pt} $$
and  the $\infty$-category of $F/K$-partition Lie algebroids\vspace{-2pt}   $$\left(\mathfrak{g} \xrightarrow{\rho} L_{F/K}^\vee[1]\right)\vspace{-2pt} $$
 satisfying   the following conditions:\vspace{3pt}
\begin{enumerate}
\item Injectivity:\vspace{2pt} the anchor map $\rho$ induces an injection $\pi_1(\mathfrak{g}) \hookrightarrow \pi_1(L_{F/K}^\vee[1]) \cong \Der_K(F)$.
\item  Vanishing: $\pi_k(\mathfrak{g})=0$ for $k\neq 0,1$.\vspace{2pt}
\item Balance:\vspace{2pt} $\dim_F(\pi_0(\mathfrak{g})) = \dim_{F}(\pi_{1}(\mathfrak{g}))<\infty$.
\vspace{2pt}
\end{enumerate}

\noindent The partition\vspace{1pt} Lie algebroid $\mathfrak{gal}_{F/K}(E) = \mathfrak{D}(E)$  corresponding to an intermediate field $E$ has underlying object $(L_{F/E}^\vee[1]  \rightarrow L_{F/K}^\vee[1])\in (\Mod_F)_{/L_{F/K}^\vee[1]},$
while the  field $F^{\mathfrak{g}}$ corresponding to a partition Lie  algebroid satisfying $(1)-(3)$  is  given by its Chevalley--Eilenberg \mbox{complex $C^\ast(\mathfrak{g})$, cf.\     \Cref{constructadjunction}.}
\end{theorem}\ \vspace{-15pt}
\begin{remark}
Note that we in particular assert that the full subcategory of $\Alg_{\Lie_{F/K}^\pi}$ spanned by all objects satisfying conditions $(1)-(3)$  in \Cref{mainintro} is the nerve of an  \vspace{2pt} ordinary category, which is in \mbox{fact a poset.}
\end{remark}

 {The proof of the correspondence proceeds as follows.  First we define partition Lie algebroids as algebras over a monad coming from an adjunction in \Cref{threeone}.  We then show that after restriction to certain subcategories, the adjunction induces an equivalence between certain partition Lie algebroids and intermediate complete local Noetherian objects in \Cref{KDsec}.  Finally we determine when these are fields using the conditions (1), (2), and (3) in \Cref{sec:fundamental_theorem}.}

 {By applying the theorem to the generic point of a variety, we immediately get a classification of purely inseparable morphisms of normal varieties.}

\begin{corollary}\label{geometric_cor}
 {Let $X$ be a normal variety over a perfect field $k$, with function field $F$.  Then purely inseparable $k$-morphisms to a normal variety $\pi:X\to Y$ of exponent at most $n$ are in bijection with $F/F^{p^n}$-partition Lie algebroids satisfying the conditions (1), (2) and (3) from \Cref{fundamental_theorem}.}
\end{corollary}

We will now record several key properties of the monad $\LieAlgd^\pi_{F/K}$:

\begin{construction}[Partition Lie algebroids]\label{maincons} \  
\begin{enumerate}
\item The functor $\LieAlgd^\pi_{F/K}$ commutes   with 
 filtered colimits and
geometric realisations.  
\item For any simplicial commutative $K$-algebra $B$ with a map to $F$, the basic
arrow $$\left(L_{F/B}^\vee[1] \rightarrow L_{F/K}^\vee[1]\right)$$
can be equipped with a canonical $F/K$-partition Lie algebroid structure $\mathfrak{D}(B)$.

\item The forgetful\vspace{1pt}  functor $\Alg_{\LieAlgd^\pi_{F/K}}  \rightarrow (\Mod_K)_{/L_{F/K}^\vee[1]}$
sending $(\mathfrak{g} \rightarrow L_{F/K}^\vee[1])$ to the underlying object in 
$(\Mod_K)_{/L_{F/K}^\vee[1]}$ 
lifts  {canonically} to a sifted-colimit-preserving 
functor
 $$U: \Alg_{\LieAlgd^\pi_{F/K}} \rightarrow (\Alg_{\Lie^\pi_{K}})_{/L_{F/K}^\vee[1]},$$
where $L_{F/K}^\vee[1]$  is the $K$-partition Lie algebra of \Cref{firstex} (2).

\item The fibre functor $ \fib: \Alg_{\LieAlgd^\pi_{F/K}}\rightarrow (\Mod_F)_{L_{F/K}^\vee/}$ sending $(\mathfrak{g} \xrightarrow{\rho} L_{F/K}^\vee[1])$ to $\fib(\rho)$   lifts   {canonically} to a sifted-colimit-preserving 
functor
$$ \Alg_{\LieAlgd^\pi_{F/K}} \rightarrow (\Alg_{\Lie^\pi_{F}})_{L_{F/K}^\vee/}, $$
where $L_{F/K}^\vee$  is the $F$-partition Lie \vspace{3pt} algebra of \Cref{firstex} (1).\\ Very informally,\vspace{4pt} the anchor map $\rho$ measures the failure of
the bracket on 
$\mathfrak{g} $ to be $F$-linear.

\item  {
We say that an object $(V\xrightarrow{} L_{F/K}^\vee[1]) \in (\Mod_{F})_{/L^\vee_{F/K}[1]}$ has  a \textit{vanishing anchor map}
if the associated map $(V\xrightarrow{} L_{F/K}^\vee[1])$ in $\Mod_F$   is nullhomotopic (in $\Mod_F$). }

Given an object $(V\xrightarrow{0} L_{F/K}^\vee[1]) \in (\Mod_{F})_{/L^\vee_{F/K}[1]}$ with vanishing anchor map, there is an equivalence
$$\LieAlgd^\pi_{F/K}(V\xrightarrow{0} L_{F/K}^\vee[1])  \ \simeq  \  ( \liep(V) \rightarrow L_{F/K}^\vee[1]).$$ 

Note that any \vspace{0pt}   $(W\rightarrow L_{F/K}^\vee[1])$ in $(\Mod_F)_{/L_{F/K}^\vee[1]}$ is the 
geometric realisation of a simplicial diagram  of   objects \mbox{$ (V\xrightarrow{0} L_{F/K}^\vee[1])$} with vanishing anchor \vspace{2pt}   \mbox{maps and $V$ coconnective.}

\end{enumerate}
\end{construction}

\begin{remark}
It might be surprising that nonzero maps can be colimits of zero maps. However,   this phenomenon is standard in homotopy theory. For instance, the identity map on $S^2$ appears as a (homotopy) pushout of a map from the diagram $\ast \leftarrow S^1 \rightarrow \ast$ to the constant \mbox{diagram 
$S^2\leftarrow S^2 \rightarrow S^2$}, in spite of the maps $\ast \rightarrow S^2$ and $S^1 \rightarrow S^2$ being nullhomotopic.\vspace{+2pt} 
\end{remark}

A finite purely inseparable field extension $F/K$ is simple precisely if $$\dim_F(\Omega^1_{F/K})=1.$$  
Using this, we show that modular extensions can be characterised using their partition Lie algebroids. 

\begin{proposition}
	Let $F/K$ be a finite purely inseparable extension. 
	Then $F/K$ is modular precisely if there  are finitely many $F/K$-partition Lie algebroids 
	$$\rho_i:\mathfrak{g}_i\to L_{F/K}^\vee[1]$$
	such that the following conditions hold:\vspace{3pt}
	\begin{enumerate}
		\item each $ \mathfrak{g}_i $ satisfies \vspace{2pt}conditions $(1)-(3)$ of Theorem \ref{fundamental_theorem};
		\item  $\dim_F(\pi_0(\fib(\rho_i)))=1$ for each $i$; \vspace{2pt}
		\item the  \vspace{5pt}canonical map $L_{F/K}^\vee\rightarrow \oplus_i\fib(\rho_i)$ is an equivalence in $\Mod_F$.
	\end{enumerate} 
\end{proposition}

Our construction  of $F/K$-partition Lie algebroids   is higher categorical in nature,  but one can also construct simplicial-cosimplicial models by generalising \cite[Theorem 5.33]{brantner2021pd}  to Lie algebroids. 
In our   forthcoming paper  \cite{instance},   we will  study these concrete models in more detail, and  explore geometric applications of our Galois correspondence to algebraic foliations and Brauer groups.\\

\subsection*{Acknowledgments} 
It is a pleasure to thank Omar Antol\'{i}n--Camarena, 
Luc Illusie, Minhyong Kim, and Joost Nuiten for inspiring conversations, as well as 
Bhargav Bhatt, who introduced  the  two coauthors at MSRI, where they were supported by the National Science Foundation under Grant No.\ DMS-1440140 during the parallel Spring 2019 semester programmes  on Derived Algebraic Geometry and on Birational Geometry and Moduli Spaces, respectively.  We also wish to thank the anonymous referee for their helpful comments on an earlier version of this paper.

This work was supported by a grant from the Simons Foundation (850684, JW).

The first author wishes to thank Merton College and the Mathematical Institute at Oxford University for their support, while the second author wishes to thank Princeton \mbox{University for its support.}

\newpage

\section{Preliminaries}
To set the stage, we  briefly review some   basic notions of derived \mbox{algebraic geometry,} \mbox{including} chain complexes, simplicial commutative rings, and Andr\'{e}--Quillen's
cotangent \mbox{complex formalism.}
We will not attempt to give  a comprehensive treatment -- the aim of this section is merely to introduce the reader to several necessary ideas, and provide references which  allow them to find the details,    in particular  within the references \cite{lurie2009higher}, \cite{lurie2011derivedX}, \cite[Chapter 25]{lurie2016spectral}, and \cite{lurie2014higher}.

The language of $\infty$-categories, also known as quasi-categories, provides an essential tool for us, which we  will use freely. These objects were first defined by Boardman and Vogt \cite{MR0420609},  explored further by Joyal \cite{MR1935979},  and  treated comprehensively in the seminal  work of Lurie \cite{lurie2009higher}.  In particular, we will use the terms ``limit" and ``colimit"   in the $\infty$-categorical sense. Hence, these correspond to   ``homotopy limits" and ``homotopy colimits" in the more \vspace{-2pt} classical, \mbox{model categorical sense.}

\subsection{Chain complexes\vspace{-2pt}}
Let $R$ be an ordinary commutative ring.
\begin{notation}[Chain complexes] We write $\mod_R$ for the derived $\infty$-category \mbox{of the ring $R$}. \mbox{Its objects} are represented by chain complexes  of $R$-modules  $ \ldots \rightarrow M_1 \rightarrow M_0 \rightarrow M_{-1} \rightarrow \ldots$ (cf.\ \cite[Definition 1.3.5.8]{lurie2014higher}), or   equivalently by $R$-module spectra \cite[Remark 7.1.1.16]{lurie2014higher}).\vspace{-2pt} 
\end{notation}
Given $M\in \Mod_R$, let $\pi_i(M) $ be the $i^{th}$ homology group of the corresponding chain  complex, 
or equivalently the  $i^{th}$ homotopy group  of the corresponding $R$-module spectrum.\vspace{-1pt}

\begin{remark}[Basic properties of $\Mod_R$]\label{basicprop}
 We recall several basic facts from \cite[Chapter 1, 7]{lurie2014higher}.
\begin{enumerate}
\item The homotopy category $h\Mod_R$ of $\Mod_R$ is equivalent to the classical derived \mbox{category of $R$.} In particular, $\Mod_R$ should not be confused with the category of ordinary $R$-modules. However, $\Mod_R$ can be equipped with a natural $t$-structure, the heart of which recovers ordinary $R$-modules as chain complexes concentrated in \mbox{degree zero (cf.\ \cite[Definition 7.1.1.13]{lurie2014higher}).}
\item The full subcategory $\Mod_{R,\geq 0}$ of connective objects for this $t$-structure consists \mbox{of all $M$ with} $\pi_i(M)=0$ for $i<0$, i.e.\ by all chain complexes with vanishing \mbox{homology in negative degrees.}
\item In fact,  $\Mod_{R,\geq 0}$ can   be
obtained by freely adjoining sifted colimits, 
i.e.\ filtered colimits and geometric realisations, 
 to the  ordinary category of   finitely generated free $R$-modules. To this end, we use  the $\mathcal{P}_\Sigma$-construction, which sends an $\infty$-category  $\mathcal{C}$ to the  $\infty$-category $\mathcal{P}_\Sigma(\mathcal{C})$ of finite-product-preserving functors from $\mathcal{C} ^{\op}$ to   \mbox{spaces,  c.f.\    \cite[Section 5.5.8]{lurie2009higher}.}
\item 
The $\infty$-category $\mod_R$ admits a symmetric monoidal structure denoted by $\otimes_R$ (cf.\ \cite[Section 4.5.2]{lurie2014higher}), which  computes the  derived  tensor product of chain complexes over $R$, and is traditionally  denoted by $\otimes^L_R$.  This product preserves small colimits in each entry. 
\end{enumerate}
\end{remark} 
We will often need two finiteness properties of chain complexes. For simplicity, we will only spell them out in the generality needed for our later arguments.\vspace{-2pt} 
\begin{definition}[Finiteness properties in $\Mod_R$] \label{modulefinite}
Let $R$ be an ordinary \vspace{1pt} Noetherian \mbox{commutative ring.} A chain\vspace{-2pt} complex $M\in \Mod_R$ is said to be 
\begin{enumerate}
\item \textit{perfect} if it can be represented by a finite complex of finitely generated free $R$-modules  {\mbox{(cf.\ \cite[Definition 7.2.4.1]{lurie2014higher})}};
\item \textit{almost perfect} if it is bounded below and each $\pi_i(M)$ is a finitely generated $R$-module  {\mbox{(cf.\ \cite[Definition 7.2.4.10 and Proposition 7.2.4.17]{lurie2014higher} or} also \cite[I. 2]{berthelot225seminare} or \cite[066E]{stacks-project}, where this notion  is called `pseudo-coherent'});
\item \textit{of finite type} if each $\pi_i(M)$ is a finitely generated $R$-module. 
\end{enumerate} 
Write $\Pperf_R, \APperf_R, \Mod_R^{\ft} \subset \Mod_R$ for the full subcategories spanned by these  families of modules.
\end{definition} 

\begin{remark} One can also develop this theory for $R$ equal to the sphere spectrum, in which case we  recover the symmetric monoidal $\infty$-category $\Sp$ of spectra. Any ordinary ring gives rise to a commutative algebra object in $\Sp$, its Eilenberg-MacLane spectrum, and $\Mod_R$ is equivalent to the $\infty$-category of modules over this $\EE_\infty$-ring. We will not need this perspective in our work.
\end{remark}

\subsection{Simplicial commutative rings}
The affine objects in derived algebraic geometry are given by simplicial commutative rings. Their
homotopy theory  was first studied by Quillen \cite{Quillen1}, who constructed a cofibrantly generated  model structure on this category. The  $\infty$-category obtained by inverting weak equivalences admits a concise presentation, c.f.\  \cite[Definition 25.1.1.1]{lurie2016spectral}: 
\begin{definition}[Simplicial Commutative Rings]
Let $R$ be an ordinary commutative ring, and write $\Poly_R$ for (the nerve of) the category of finitely generated \mbox{polynomial $R$-algebras $R[x_1,\ldots, x_n]$.}
 The \emph{$\infty$-category of simplicial commutative $R$-algebras} is given by
 $$\SCR_R := \mathcal{P}_\Sigma(\Poly_R ^{\op}),$$
that is, by  the $\infty$-category of all finite-product-preserving functors from $\Poly_R ^{\op}$ to \mbox{spaces.}
\end{definition}	
A detailed $\infty$-categorical treatment  of simplicial commutative rings is given in \mbox{\cite[Chapter 25]{lurie2016spectral}.} 
\begin{remark}[Basic properties of $\SCR_R$] We review several basic facts.
	\begin{enumerate}
		\item The $\infty$-category $\SCR_R$ is obtained from $\Poly_R$ by \mbox{formally} adjoining sifted colimits, i.e.\ filtered colimits and geometric realisations (c.f.\    \cite[Section 5.5.8]{lurie2009higher}). It is presentable, and the Yoneda embedding 
$\Poly_R\hookrightarrow \SCR_R$ preserves coproducts. The image of this embedding consists of compact objects, which generate $\SCR_R$ under  sifted colimits.
		\item Every  $B\in \SCR_R$
 has an underlying (connective) \mbox{spectrum,} which can in fact be equipped with the
 structure of an $\EE_\infty$-$R$-algebra in a canonical way.
\item  Hence, we can take the homotopy groups $\pi_*(B)$ of any $B\in\SCR_R$, and 
each $\pi_i(B)$ is a module over the ordinary commutative ring $\pi_0(B)$. If we model $B$ by an ordinary simplicial commutative $R$-algebra, $\pi_i(B)$ is the $i^{th}$ homology of the chain complex obtained by applying \mbox{the Dold-Kan correspondence to $B$,} thought of as a simplicial $B$-module.
\item If $R=\ZZ$, we  refer to  $\SCR := \SCR_\ZZ$ as the $\infty$-category of simplicial commutative rings; 
for any ring $R$, there is a natural forgetful functor $\SCR_R \rightarrow \SCR$. 
\item More generally, given a map of rings $R\rightarrow S$, there is a   forgetful functor \mbox{$\SCR_S \rightarrow \SCR_R$}. Its left adjoint  is computed by $R \otimes_S (-)$ on the level of modules.
	\end{enumerate}
\end{remark}

\begin{notation}
Given $S\in \SCR_R$, write $\SCR_{R//S} := (\SCR_R)_{/S}$ for the overcategory    (cf.\ \cite[Section 1.2.9]{lurie2009higher})  consisting of simplicial commutative $R$-algebras with a map to $S$. When $R=S$, we obtain the $\infty$-category	of augmented simplicial commutative $R$-algebras \vspace{-2pt}$\SCR_{R}^{\aug} := \SCR_{R//R}$.
\end{notation}

The following class of  simplicial commutative rings will play a key role in our arguments:
\begin{definition}[Complete local Noetherian objects]\label{cNdef}
	A simplicial commutative ring  $B\in \SCR$ is \textit{complete local Noetherian} if   $\pi_0(B)$ is  a complete local Noetherian ring  and $\pi_i(B)$ is a finitely generated $\pi_0(B)$-module for all $i\geq 0$.
	Write $\SCR^{\cN}\subset \SCR$,  $\SCR_R^{\cN}\subset \SCR_R$, \mbox{$\SCR_{R//S}^{\cN}\subset  \SCR_{R//S}$}  for the full subcategories spanned by  \vspace{-2pt}  complete local \mbox{Noetherian rings.}
\end{definition}

\subsection{Algebraic Andr\'{e}--Quillen homology}
The homology of simplicial commutative rings was introduced by Andr\'{e} \cite{andre1974homologie} and Quillen \cite{quillen1970co}, and may be thought of as the nonabelian derived functor of the construction which sends an augmented algebra to its indecomposables.

To give a more formal construction, we first introduce its right adjoint:
\begin{definition}[Trivial square-zero extensions] Let $R$ be an ordinary commutative ring. The  \textit{trivial square-zero extension functor}  $\sqz_{R}$ is the unique sifted-colimit-preserving functor  \vspace{-2pt}
$$\sqz_{R}:\Mod_{R,\geq 0}\to \SCR^{\aug}_{R} \vspace{-0pt}$$
sending a
finite free $R$-module $N$ to the chain  complex $R\oplus N$ with multiplication \[(r_1,n_1)\cdot (r_2,n_2) = (r_1 r_2, r_1 n_2 + r_2 n_1).\] Note that  the underlying chain complex of $\sqz_R(M)$ is always given by $R\oplus M$.\vspace{4pt}

 {This construction can be extended to a trivial square-zero extension functor for connective  modules over   simplicial commutative rings, see \cite[Section 25.3.1]{lurie2016spectral}.  To this end, we consider the $\infty$-category $\mathrm{SCRMod}^{\mathrm{cn}}$ whose objects are pairs $(R,M)$ where $R$ is a simplicial commutative ring and $M$ is a connective $R$-module (cf.\ \cite[Notation 25.2.1.1]{lurie2016spectral}), and the full subcategory 
$\mathcal{C}\subset \mathrm{SCRMod}^{\mathrm{cn}} $ spanned by pairs consisting of a polynomial ring $A=\mathbb{Z}[x_1,..,x_n]$ and a  finite free $A$-module.   
By \cite[Proposition 25.2.1.2]{lurie2016spectral}, there is an equivalence  $\mathrm{SCRMod}^{\mathrm{cn}}\simeq \mathcal{P}_{\Sigma}(\mathcal{C})$ between the sifted completion of $\mathcal{C}$ (cf.\ \cite[Definition 5.5.8.8]{lurie2009higher})
 and $\mathrm{SCRMod}^{\mathrm{cn}}$.
Using the universal property of the $\mathcal{P}_{\Sigma}$-construction in \cite[Proposition 5.5.8.15]{lurie2009higher}, one  obtains a sifted-colimit-preserving functor $\mathrm{SCRMod}^{\mathrm{cn}} \rightarrow \SCR$, $(R, M) \mapsto \sqz_R(M)$  which extends the classical square-zero extension\vspace{4pt} functor.}

More generally, let $R\rightarrow S$ be a map of ordinary commutative rings, the    \textit{relative trivial square-zero} extension functor $\sqz_{R//S}:\Mod_{S,\geq 0}\to  {\SCR_{R//S}}$ sends  $M\in \Mod_{S,\geq 0}$ to $\sqz_S(M)$, considered as a simplicial commutative $R$-algebra via restriction along the map $R\rightarrow S$.

\end{definition}
Here, we have used the universal property of $\Mod_{R,\geq 0}$ discussed in \Cref{basicprop}(3).
\begin{notation}
We will often write $\sqz_R(M) = R\oplus M$ and $\sqz_{R//S}(M) = S \oplus M$.
\end{notation}
We can now introduce  {Andr\'{e}--Quillen's homology functor:
 \begin{definition}[Cotangent fibre] \label{cotfib}
For $R\in \SCR$, the  \mbox{\textit{cotangent fibre functor}}
$$\cot_R: \SCR_{R}^{\aug} \rightarrow \Mod_{R, \geq 0} $$
is given by the left adjoint to the trivial square-zero extension functor $\sqz_R$.

Given a map $R\rightarrow S$, the left adjoint of $\sqz_{R//S}$ gives a relative 
version of this functor \mbox{denoted by}
$$\cot_{R//S}:  {\SCR_{R//S}} \rightarrow \Mod_{S, \geq 0}.$$
\end{definition}
Since $\sqz_{R//S}$ is a composite of right adjoints, we can write $\cot_{R//S}$ as a composite of left adjoints $$\cot_{R//S}(B) \simeq \cot_S(S \otimes_R B) .$$

 {We will describe cotangent fibres in terms of cotangent complexes in the following \Cref{acc}.}

 \begin{remark}
The cotangent fibre functor above is often decorated with a subscript $(-)_\Delta$, to indicate that 
we are working over  simplicial commutative rings, rather than $\EE_\infty$-rings. As we only use one of these versions, we will drop this subscript from our notation.\end{remark}

The following result links two finiteness conditions for modules in \Cref{modulefinite}  and rings in \Cref{cNdef}, and 
 follows from 
\cite[Proposition 3.1.5]{lurie2004derived} and 
\cite[Proposition 3.2.14]{lurie2004derived}:
\begin{proposition}\label{cnoethcot} 
	If $A \in \SCR_F^{\aug}$ is Noetherian, then $\cot(A) \in \Mod_{F,\geq 0}^{\ft}$ is connective  of \mbox{finite type.}
\end{proposition}

\subsection{The algebraic cotangent complex}\label{acc}
The cotangent complex formalism is of central importance in  deformation theory, as 
was illustrated, for example, in the pioneering \mbox{work of  Illusie \cite{Illusie_1971}.}
  Given a ring map $A \rightarrow B$, the (algebraic) cotangent complex is a derived version of the module of K\"{a}hler differentials $\Omega^1_{B/A}$, and its construction parallels the classical definition of $\Omega^1_{B/A}$.  For more details, we refer to the modern treatment in  \cite[Sections 25.3.1, 25.3.2]{lurie2016spectral}.

\begin{definition}[Derivations]
	For a simplicial commutative ring $R$ and connective $R$-module $M$, we define the \emph{space of derivations of $R$ into $M$} as the following mapping space:
	$$\Der(R,M)=\Map_{({\SCR})_{/R}}(R,R\oplus M).$$
\end{definition}

One can prove that there  exists an $R$-module $L_R$ and a universal derivation $\eta\in \Der(R,L_R)$ such that  for any connective $R$-module $M\in \Mod_{R,\geq 0}$, the natural map  $\Map_{\Mod_R}(L_R, M)\simeq\Der(R,M)$ is an  equivalence. The pair $(L_R, \eta)$ is uniquely 
determined up to equivalence.
\begin{definition}[The algebraic cotangent complex]\

\begin{enumerate}
\item 
The $R$-module $L_R  $ is called the   {\textit{the absolute (algebraic) \emph{cotangent complex} of $R$}.}
\item
For a morphism of simplicial commutative rings $R\to S$, the \emph{relative (algebraic) cotangent complex} $L_{S/R} $ is the cofibre of the natural map $S\otimes_R L_R\to L_S$.
\end{enumerate}
\end{definition}
 The relative cotangent complex also satisfies  a universal property; for any $S$-module $M$,  we have
	$$\Map_{\Mod_S}(L_{S/R},M)\simeq\Map_{(\SCR_{R})_{/S}}(S,S\oplus M).$$

 {Unraveling the definitions, we see that given a  morphism $R\rightarrow S$ and some \mbox{$B\in \SCR_{R//S}$, we have}
$$\cot_{R//S}(B) \simeq S \otimes_B L_{B/R}. $$
Taking $R=S$, we deduce  that for  augmented $R$-algebra $B\in \SCR_R^{\aug}$, we have 
\mbox{$ \cot_R(B) \simeq R \otimes_B L_{B/R}.$} }

\begin{remark} The algebraic cotangent complex is often decorated with a superscript $(-)^{\aaalg}$  to distinguish it from the topological cotangent complex of the underlying $\EE_\infty$-rings.
\end{remark}

The most important computational  tool in dealing with the cotangent complex is the fundamental cofibre sequence, which extends  the classical relative cotangent sequence:
\begin{proposition}[The fundamental cofibre sequence]\label{lem:triangle}
	Given maps of simplicial commutative rings \mbox{$f:A \rightarrow B$} and $g:B\rightarrow C$, there is  a canonical cofibre sequence in $\Mod_C$ given by 
	$$C\otimes_{B} L_{B/A} \rightarrow L_{C/A} \rightarrow L_{C/B}. $$
\end{proposition}

 {Note that the cofibre sequence allows us to write the functor $\cot_{R//S}$ in the equivalent form 
\[\cot_{R//S}(B)=\cofib(L_{S/R}[-1]\to L_{S/B}[-1])\] as explained in }\cite[Remark 1.1.1.7]{lurie2014higher}.

This construction of the cofibre sequence is functorial in the following sense:

\begin{proposition} \label{threemaps} {
There is a functor of $\infty$-categories  $\SCR_{A//C} \rightarrow \Fun(\Delta^2, \Mod_C)$ sending an object $(A \rightarrow B \rightarrow C) \in \SCR_{A//C}$ to the 
canonical cofibre sequence $C\otimes_{B} L_{B/A} \rightarrow L_{C/A} \rightarrow L_{C/B} $ from \Cref{lem:triangle} and a morphism 
$ (A \rightarrow B_1 \rightarrow B_2 \rightarrow C) \in \Fun(\Delta^1, \SCR_{A//C})$ to the canonical map of cofibre sequences  
$$\xymatrix{
C\otimes_{B_1} L_{B_1/A}  \ar[r] \ar[d] &  L_{C/A}  \ar[d]  \ar[r] & L_{C/B_1} \ar[d]\\ 
C\otimes_{B_2} L_{B_2/A}  \ar[r] & L_{C/A}  \ar[r]   & L_{C/B_2} }. $$}

\end{proposition}

 {\begin{proof} {
Write $\Fun'(\Delta^1,  \SCR_{A//C}) \subset \Fun(\Delta^1,  \SCR_{A//C})$ for the full subcategory of arrows sending $1 \in \Delta^1$ to $(A \xrightarrow{ } C \xrightarrow{\id} C) \in \SCR_{A//C}$.  This object is final,  which means that 
evaluation at $0\in  \Delta^1$  provides an equivalence 
$\Fun'(\Delta^1,  \SCR_{A//C}) \simeq  \SCR_{A//C}$.  Picking an inverse from a contractible space of choices gives rise to a functor 
$ \SCR_{A//C} \rightarrow \Fun(\Delta^1,  \SCR_{A//C}) $ sending 
$(A \xrightarrow{} B \rightarrow C) $ to $   (A \xrightarrow{ } B \xrightarrow{}C \xrightarrow{\id} C) $.}

 {Postcomposing it with the cotangent fibre functor $\cot_{A//C}$ from \Cref{cotfib}
gives the  auxiliary  functor $\SCR_{A//C} \rightarrow  \Fun(\Delta^1,  \Mod_C)$ sending $(A \xrightarrow{} B \rightarrow C)$ to the arrow $C \otimes_B L_{B/A} \rightarrow L_{C/A}$.}

 {Finally,  let us write $ \Fun'(\Delta^2, \Mod_C) \subset  \Fun(\Delta^2, \Mod_C)$ for the full subcategory spanned by all cofibre sequences.
Restriction to $\Delta^1 \simeq \Delta^{\{0,1\}} \subset \Delta^2$ defines an equivalence 
$ \Fun'(\Delta^2, \Mod_C)\xrightarrow{\simeq} \Fun(\Delta^1, \Mod_C) $ sending a cofibre sequence $(M \rightarrow M' \rightarrow M'')$ to $(M \rightarrow M')$.}

 {Postcomposing with its inverse,  chosen again from a contractible space of choices as in \cite[Remark 1.1.1.7]{lurie2014higher},  we obtain the desired functor 
$\SCR_{A//C} \rightarrow \Fun(\Delta^2,\Mod_C) $ sending $(A \rightarrow B \rightarrow C)$ to $(C \otimes_B L_{B/A} \rightarrow L_{C/A} \rightarrow L_{C/B}),$
where we have  identified the cofibre of $C \otimes_B L_{B/A} \rightarrow L_{C/A}$ with $L_{C/B}$ using \Cref{lem:triangle}.}
\end{proof}}

\subsection{Formal moduli problems}Any reasonably geometric deformation problem over a field $K$ gives rise to a formal moduli problem. To formalise this notion, we need a family of augmented simplicial commutative $K$-algebras corresponding to  the derived infinitesimal thickenings of $\Spec(K)$:
\begin{definition}[Artinian $K$-algebras] \label{artinian}
A simplicial commutative $K$-algebra $A$ is \textit{Artinian} if 
\begin{enumerate}
 \item   {The $K$-algebra $\pi_0(A)$ is an Artinian ring with residue field, as a $K$ algebra, being $K$;}
\item  {The $K$-vector space $\pi_\ast(A) = \oplus_{n=0}^\infty\pi_n(A)$} is  finite-dimensional.
\end{enumerate}
We write $\SCR^{\art}_K \subset \SCR_K^{\aug}$ for the full subcategory spanned by all  {(simplicial) Artinian $K$-algebras $A$, along with the canonical augmentation map given by the composite $A\to \tau_{\leq 0}(A)=\pi_0(A)\to K$.}
\end{definition}
We can now recall Lurie's higher categorical framework for  formal deformation \mbox{functors, cf.\ \cite{lurie2011derivedX}:}
\begin{definition}  \label{fmpdef}
	A \emph{formal moduli problem} is a functor $X:\SCR^{\art}_K\to \mathcal{S}$ such that
	\begin{enumerate}
		\item $X(K)\simeq \ast$ is contractible;
		\item  For any pullback square \vspace{-15pt}
		$$
		\begin{tikzcd}
		A_3\ar[r]\ar[d] & A_2\ar[d]\\
		A_1\ar[r] & A_0\vspace{-5pt}
		\end{tikzcd}
		$$
		where $\pi_0(A_2)\to \pi_0( A)$ and $\pi_0(A_1)\to \pi_0(A)$ are surjective,
	applying $X$ gives a pullback\vspace{-3pt}
		$$
		\begin{tikzcd}
		X(A_3)\ar[r]\ar[d] & X(A_2)\ar[d]\\
		X(A_1)\ar[r] & X(A_0).
		\end{tikzcd}\vspace{-3pt}
		$$
	\end{enumerate}	
	We denote the $\infty$-category of formal moduli problems over $K$ by $\mathrm{Moduli}_{K}$.
\end{definition}

\begin{example}\label{ex:functor_points}
	Let $A$ be a commutative ring which is an augmented $K$-algebra.  Then $A$ gives a formal moduli problem via the functor of points:  $\Spf(A):\SCR^{\art}_{K//K}\to \mathcal{S}, \ R\mapsto \Map_{\SCR_{K//K}}(A,R)$.\vspace{-5pt}
\end{example}

In characteristic zero,   Lurie \cite{lurie2011derivedX} and Pridham \cite{pridham2010unifying}  
showed that formal moduli problems  are controlled by  differential graded Lie algebras.
Partition Lie algebras were introduced in \cite{brantner2019deformation} to generalise this statement to base fields of arbitrary characteristic.
Note that the link between Lie algebras and formal deformation theory only becomes fully realised if one works in the   setting of derived algebraic geometry.

We will now describe the relation between  formal moduli problems and  Lie \mbox{algebras in more detail. } To begin with, recall that 
given a formal moduli problem $X$ over a field $K$, we can associate its tangent fibre \mbox{$T_X\in \Mod_K$,} which, as a spectrum, satisfies 
$$\Omega^{\infty-n}(T_X) =X(K\oplus K[n]). $$ Here $K\oplus K[n]$ is the trivial square-zero extension of $K$ by a copy of $K$ in
(homological) degree $n$.

In fact, $T_X\in \Mod_K$ is characterised by a natural equivalence 
\begin{equation} \label{functorialtangent} \Map_K(V^\ast, T_X) \simeq  X( K \oplus V),  \end{equation}
where $V$ varies over perfect connective $K$-modules , and  $(-)^\ast$ denotes $K$-linear duality.

The key observation is that $T_X$ can be endowed the structure of a \textit{partition Lie algebra}, that is,  an algebra over a certain 
monad $\Alg_{\Lie_{K}^\pi}$ on the derived $\infty$-category $\Mod_K$ of chain \mbox{complexes over $K$. }
We will describe this monad in more detail in \Cref{extending} below.
The main result of \cite{brantner2019deformation} uses this construction to generalise  the Lurie--Pridham theorem to arbitrary characteristics:

\begin{theorem}[\cite{brantner2019deformation}, Theorem 1.11] \label{plieeq}
	If $K$ is a field,  the functor $\mathrm{Moduli}_{K} \rightarrow \Mod_K$ sending a formal moduli problem $X$ to its tangent fibre $T_X$ refines to an 
 equivalence of $\infty$-categories $$\mathrm{Moduli}_{K}\simeq \Alg_{\Lie_{K}^\pi}.$$ 
\end{theorem}
  
\subsection{ {Extending monads and functors}}\label{extending}
We will briefly outline the higher categorical construction of the partition Lie algebra monad $\Lie_K^\pi$ presented in 
 \cite{brantner2019deformation}, which we will  generalise to the case of Lie algebroids in \Cref{threeone} of the main text.
For an alternative, more explicit, construction of partition Lie algebras using point set models, we refer to  \mbox{\cite[Theorem 5.33, Construction 5.34]{brantner2021pd}.}
 
To construct the monad $\Lie_K^\pi$, we first note that the (contravariant) tangent fibre functor $\SCR_{K//K}^{\op} \rightarrow \Mod_{K,\leq 0}$, 
$A \mapsto \cot(A)^\vee= (k \otimes_A L_{A/k})^\vee$ from  augmented simplicial commutative $K$-algebras to  chain complexes over $K$ is part of an adjunction. \vspace{4pt}

The associated monad $T$ on $\Mod_{K, \leq 0}$ sends
$V \in \Mod_K$ to $\displaystyle L^\vee_{K/K \oplus V^\vee}[1] \simeq (K \myotimes{K \oplus V^\vee} L_{K \oplus V^\vee/K})^{\vee}$. 

However,   $T$ is \textit{not} the right monad for the purposes of deformation theory,
as it does not preserve sifted colimits. 
To overcome this obstacle, we replace $T$ by a more well-behaved monad $\Lie_K^\pi$.

\begin{construction}[Partition Lie algebra monad]\label{Plamonad}
We briefly outline the construction of $\Lie_K^\pi$, which appears,  in a more abstract language,  in the proof of \cite[Corollary 5.46]{brantner2019deformation}.
\begin{enumerate}
\item First, we check that  the monad $T$ preserves the full subcategory $\Mod_{K,\leq 0}^{\ft} \subset \Mod_K$ of chain complexes $V$ which are coconnective and of finite type (cf.\ e.g.\  \cite[Proposition 3.2.14]{lurie2004derived}).
Hence  $T|_{\Mod_{K,\leq 0}^{\ft}}$ {acquires the structure of a monad}.

The proof of \cite[Proposition 5.49]{brantner2019deformation} gives a  description of this restriction $T|_{\Mod_{K,\leq 0}^{\ft}}$: if 
$V^\bullet$ is a cosimplicial $K$-module whose associated chain complex $\Tot(V^\bullet)$ is  \mbox{of finite type, then}
 $$T|_{\Mod_{K,\leq 0}^{\ft}}(\Tot(V^\bullet))  \simeq \bigoplus_{n} \Tot \left(\widetilde{C}^\bullet(\Sigma |\Pi_n|^\diamond,K) \otimes (V^\bullet)^{\otimes n}\right)^{\Sigma_n}.$$  
Here $\widetilde{C}^\bullet(-,K)$ are  the  $K$-valued  cosimplices of the (doubly suspended) $n^{th}$  partition complex (cf.\ e.g.\ \cite{arone2018action}),  the functor $(-)^{\Sigma_n}$ takes strict fixed points,  and the tensor product is  {computed in cosimplicial $K$-modules.}

\item Using this explicit description,  it is not hard to check that the functor $T|_{\Mod_{K,\leq 0}^{\ft}}$ is
 
\begin{itemize}
\item \textit{right complete}  which   means that
the canonical map $\colim_n T(\tau_{\leq -n} V) \xrightarrow{\simeq} T(V)$ is an equivalence  for all $V\in \Mod_{K,\leq 0}^{\ft}$;
\item \textit{preserves finite coconnective geometric realisations}, which 
  means that if $V_\bullet$ is a simplicial object  in $\Mod_{K,\leq 0}^{\ft}$ which is $m$-skeletal (for some $m$) and satisfies with $|V_\bullet| \in \Mod_{K,\leq 0}^{\ft}$, then the canonical map $|T(V_\bullet)| \xrightarrow{\simeq} T(|V_\bullet|)$ is an equivalence.\end{itemize}
\item  Next, we show  in \cite[Proposition 3.16]{brantner2019deformation} 
that restriction induces an equivalence 
$$\End_\Sigma(\Mod_K) \xrightarrow{\simeq} \Fun_{\sigma}'(\Mod_{K, \leq 0}^{\ft}, \Mod_K) $$
between the full subcategory $\End_\Sigma(\Mod_K)  \subset  \End (\Mod_K)$ of sifted-colimit-preserving endofuctors  of $\Mod_K$ and the full subcategory 
$ \Fun_{\sigma}'(\Mod_{K, \leq 0}^{\ft}, \Mod_K) \subset  \Fun(\Mod_{K, \leq 0}^{\ft}, \Mod_K)$ of  right complete 
functors which preserve finite coconnective geometric realisations.

Let us  write $\End_\Sigma^{\Mod_{K,\leq 0}^{\ft}} \subset \End(\Mod_K)$ for the full subcategory of sifted-colimit-preserving endofuctors  of $\Mod_K$ which preserve $\Mod_{K,\leq 0}^{\ft}$, and let $\End_\sigma'(\Mod_{K,\leq 0}^{\ft})\subset \End(\Mod_{K,\leq 0}^{\ft})$ be the full subcategory of right complete endofunctors of 
$\Mod_{K,\leq 0}^{\ft}$ which preserve finite coconnective geometric realisations.

Then the above equivalence implies the that the following restriction functor is an equivalence as well (cf.\ \cite[Corollary 3.17]{brantner2019deformation}):
$$\End_\Sigma^{\Mod_{K,\leq 0}^{\ft}} (\Mod_K) \xrightarrow{\simeq} \End_\sigma^{'}(\Mod_{K,\leq 0}^{\ft})$$ 
\item Using this equivalence, we extend the monad $T|_{\Mod_{K,\leq 0}^{\ft}}$ to  obtain the  monad 
 $\Lie_K^\pi$ on $\Mod_K$.
This monad  $\Lie_K^\pi$ preserves
  filtered colimits and geometric realisations, and if 
$V^\bullet$ is a cosimplicial $K$-module with associated chain complex $\Tot(V^\bullet)$, then
 $$\Lie_K^{\pi} (\Tot(V^\bullet))  \simeq \bigoplus_{n} \Tot \left(\widetilde{C}^\bullet(\Sigma |\Pi_n|^\diamond,K) \otimes (V^\bullet)^{\otimes n}\right)^{\Sigma_n}.$$   
\end{enumerate}
\end{construction}

\newpage

\newpage

\newcommand*{\myovline}[2]{\overbracket[#2][-1pt]{#1}}

\newpage

\section{Partition Lie Algebroids}\label{sec:plads}
Let $F/K$ be a finite purely inseparable field extension. Write $L_{F/K}^\vee[1]$ for the shifted   dual of its relative cotangent complex, so that $\pi_1(L_{F/K}^\vee[1]) \cong \Der_K(F)$ consists of \mbox{$K$-linear derivations of $F$.}
In this section, we will  define $F/K$-partition Lie algebroids and establish some of their key \mbox{properties,} drawing inspiration from
the partition Lie algebras introduced in \cite{brantner2019deformation}.

\subsection{Constructing partition Lie algebroids} \label{threeone} Partition Lie algebroids will consist of \vspace{-1pt} arrows $$(\mathfrak{g} \rightarrow L_{F/K}^\vee[1])$$
with additional structure. More formally, there will be a monad $\LieAlgd^\pi_{F/K}$  on the \vspace{-1pt}   \mbox{overcategory} $$\left(\Mod_F\right)_{/L_{F/K}^\vee[1]}$$ whose  algebras will be the desired $F/K$-partition Lie algebroids. These algebraic structures will later allow us to classify intermediate fields $K \subset E \subset F$. To construct our monad, we will prove:

\begin{theorem}[The $F/K$-partition Lie algebr{oid} monad]\label{maindef} \ 
\begin{enumerate}
\item The relative tangent fibre \vspace{-2pt}  $$\cot_{K//F}^\vee: \SCR_{K//F} ^{\op} \rightarrow \Mod_{F, \leq 0} \vspace{-2pt} \ \ \ \ \ \ \ \ \ \ \ \   \ \ \ \ \ \  \ \ \ \ \ \ $$ $$  \ \ \ \ \ \  \ \ \ \ \ \  \ \ \ \ \ \ \  B \mapsto   \fib(L_{F/B}^\vee[1] \rightarrow L_{F/K}^\vee[1]) $$  admits a left adjoint  $V \mapsto F\oplus V^\vee$. 
Let $\myovline{T}{1.0pt}  :\Mod_{F, \leq 0} \rightarrow \Mod_{F, \leq 0}$ be the induced monad. \vspace{3pt}

\item There is a unique monad $T$ on $\Mod_F$ which agrees with $\myovline{T}{1.0pt} $ on $\Mod_{F,\leq 0}^{\ft}$ and moreover preserves filtered colimits and geometric realisations.  Write $\mathcal{C} $ for the \mbox{$\infty$-category of $T$-algebras.}\vspace{3pt}

\item The   forgetful functor $\mathcal{C} \rightarrow \Mod_F$ factors over the fibre functor \mbox{$\fib: (\Mod_{F})_{/L_{F/K}^\vee[1]} \rightarrow \Mod_F$}
via a canonical monadic right adjoint $$\mathcal{C} \rightarrow  (\Mod_{F})_{/L_{F/K}^\vee[1]},$$ which gives rise to a   sifted-colimit-preserving monad $\LieAlgd^\pi_{F/K}$ on $(\Mod_{F})_{/L_{F/K}^\vee[1]}$.
\end{enumerate}
\end{theorem}

Assuming this result, we can introduce the main definition of this paper:

\begin{definition}[Partition Lie algebroids]\label{palgebroiddef} An \textit{$F/K$-partition Lie algebroid} is an algebra for the monad $\LieAlgd^\pi_{F/K}$ on 
$(\Mod_{F})_{/L_{F/K}^\vee[1]}$.  
 Write  $\Alg_{\LieAlgd^\pi_{F/K} }\simeq \mathcal{C} $ for the   resulting $\infty$-category.
The underlying object of a partition Lie algebroid is usually denoted by $(\mathfrak{g} \rightarrow L_{F/K}^\vee[1])$.
\end{definition}

In our proof of \Cref{maindef} and elsewhere, we will need to compute limits and colimits in undercategories (and dually overcategories).
This is  possible by the following standard result:

\begin{lemma}[Limits and colimits in undercategories] \label{limcolimunder}  
	Let $x$ be an object in  an \mbox{$\infty$-category $\mathcal{C}$.} The forgetful functor $ \mathcal{C}_{x/}\rightarrow \mathcal{C}$ 
creates  small colimits indexed by \mbox{contractible categories and all small limits.}
\end{lemma}
  \begin{proof}
		The claim about limits follows by the dual of \cite[Proposition 1.2.13.8]{lurie2009higher}.   The claim about colimits follows from \cite[Proposition 4.4.2.9]{lurie2009higher} and \cite[Proposition 2.1.2.1]{lurie2009higher}.		
	\end{proof}

To prove \Cref{maindef}, we will first establish a variant of \cite[Example 6.14]{brantner2019deformation}, which will then allow us to deduce the theorem from corresponding results in [op.cit.].

\begin{proposition}\label{basiccof}
 {There is a natural cofibre sequence of functors $\Mod_{F,\leq 0}^{\ft}\to \Mod_F$ given by}
$$ L_{F/K}^\vee \rightarrow \myovline{T}{1.0pt}  \rightarrow \Lie_{F}^\pi,\vspace{2pt} $$
for  $L_{F/K}^\vee$   the constant functor and  $\Lie_{F}^\pi$ the  partition Lie algebra monad, \mbox{cf.\ \cite[Definition 5.47]{brantner2019deformation}.}
\end{proposition}
\begin{proof}
Given $V\in \Mod_{F, \leq 0}^{\ft}$,  write $B = F\oplus V^\vee$ for \vspace{2pt} the trivial square-zero extension of $F$ by $V^\vee$, considered as an object of $\SCR_{K//F}$. Extending the  fundamental cofibre sequence in \Cref{lem:triangle} to the left, we obtain a cofibre sequence $L_{F/B}[-1]  \rightarrow (F \otimes_B L_{B/K}) \rightarrow L_{F/K}.$
It induces the asserted natural cofibre sequence after taking the $F$-linear duals, because we have \mbox{$\myovline{T}{1.0pt} (V) \simeq (F \otimes_B L_{B/K})^\vee$}   and $\Lie_{F}^\pi(V) \simeq (L_{F/B}[-1] )^\vee$ by construction of these two monads.\vspace{-1pt}

\end{proof}

\begin{proof}[Proof of \Cref{maindef} ] Statement $(1)$ follows\vspace{3pt} directly from the definition of  {$\cot_{K//F}$}. 

For $(2)$, we first show that the functor $\myovline{T}{1.0pt} $ preserves the subcategory $\Mod_{F, \leq 0}^{\ft}\subset \Mod_F$,  commutes with finite coconnective geometric realisations and is   right complete  {(cf.\ \Cref{Plamonad}(2)).}

Indeed, this follows  from the cofibre sequence in \Cref{basiccof} and  the corresponding claims for the functor $\Lie_{F}^\pi$, which are
 \mbox{established in    \cite[Corollary 5.46]{brantner2019deformation}.} We then deduce from    {the second equivalence in  \Cref{Plamonad}(3) (cf.\ \cite[Corollary 3.17]{brantner2019deformation})}   that   $\myovline{T}{1.0pt} |_{\Mod_{F, \leq 0}^{\ft}}$ extends uniquely to 
a \mbox{sifted-colimit-preserving monad\vspace{3pt} $T$ on $\Mod_F$}. 

For $(3)$, note that the cofibre sequence in \Cref{basiccof} induces a  transformation of functors\vspace{-2pt} $$ L_{F/K}^\vee \rightarrow T.\vspace{-2pt}$$
 {by applying the first equivalence in \Cref{Plamonad}(3) (cf.\ 
\cite[Proposition 3.16]{brantner2019deformation}).}
Hence, every $T$-algebra receives a canonical map from  $L_{F/K}^\vee $. We obtain a \mbox{factorisation of the forgetful} functor \mbox{$\mathcal{C} :=\Alg_T  \rightarrow \Mod_F$}   over the forgetful functor \mbox{$U: (\Mod_F)_{ L_{F/K}^\vee/}\xrightarrow{} \Mod_F$,  $( L_{F/K}^\vee \rightarrow V)\mapsto V$.}

By \Cref{limcolimunder}, the forgetful functor $U$ creates  small limits,  filtered colimits,  and  geometric realisations, because 
filtered categories and $\mathbf{\Delta} ^{\op}$ are contractible categories.
Since the forgetful functor  $\mathcal{C}   \rightarrow \Mod_F$  {is conservative} and preserves small limits and sifted colimits, 
 this implies that the same holds true for $\mathcal{C}  \rightarrow (\Mod_F)_{ L_{F/K}^\vee/}$. 
The (crude) Barr--Beck--Lurie theorem (cf.\ \cite[Theorem 4.7.0.3]{lurie2014higher}) therefore implies that $$\mathcal{C}  \rightarrow (\Mod_F)_{ L_{F/K}^\vee/}$$ is a sifted-colimit-preserving monadic right adjoint.\vspace{3pt} 
  
 {The fibre functor $\fib: (\Mod_{F})_{/L_{F/K}^\vee[1]} \rightarrow \Mod_F$,  $(f:M\rightarrow L_{F/K}^\vee[1]) \mapsto \fib(f)$ lifts, by \cite[Theorem 1.1.2.14]{lurie2014higher}, to a canonical  equivalence} $$ (\Mod_{F})_{/L_{F/K}^\vee[1]} \simeq  (\Mod_{F})_{L_{F/K}^\vee/}$$
$$ (f:M\rightarrow L_{F/K}^\vee[1]) \mapsto (L_{F/K}^\vee \rightarrow \fib(f))$$
$$ (\cofib(g) \rightarrow L_{F/K}^\vee[1])\mapsfrom (g:L_{F/K}^\vee \rightarrow N).\ \  \ $$
We  obtain the following square:\vspace{0pt}
\begin{equation}\label{square} 
\begin{gathered} $$\xymatrix{
 \mathcal{C} \ar[r] \ar[d]  &  (\Mod_{F})_{/L_{F/K}^\vee[1]} \ar[ld]^\simeq \ar[d] \\ 
(\Mod_{F})_{L_{F/K}^\vee/} \ar[r] & \Mod_F}.$$\end{gathered}
 \end{equation}

The top horizontal map is therefore a sifted-colimit-preserving monadic right adjoint, defining the sifted-colimit-preserving monad $\LieAlgd^\pi_{F/K}$ on \vspace{-3pt}the $\infty$-category $(\Mod_{F})_{/L_{F/K}^\vee[1]}$. 
\end{proof}

We now extend the natural cofibre sequence from \Cref{basiccof}:
\begin{corollary}\label{basiccofextended}
 {
There is a 
natural cofibre sequence  $$ L_{F/K}^\vee \rightarrow {T}  \rightarrow \Lie_{F}^\pi.$$}
\end{corollary}
\begin{proof}
 {Using the equivalence  
$\End_\Sigma(\Mod_K) \xrightarrow{\simeq} \Fun_{\sigma}'(\Mod_{K, \leq 0}^{\ft}, \Mod_K)$
described in \Cref{Plamonad}(3) (cf.\  \cite[Proposition 3.16]{brantner2019deformation}), we can
extend the cofibre sequence $L_{F/K}^\vee \rightarrow \myovline{T}{1.0pt}   \rightarrow \Lie_{F}^\pi $}   {from  \Cref{basiccof}.  Here, we use that the involved functors are right-complete and preserve finite coconnective geometric realisations by  \cite[Corollary 5.46]{brantner2019deformation} and \Cref{maindef}(2),  respectively.}

 {This gives a 
sequence $ L_{F/K}^\vee \rightarrow {T}    \rightarrow \Lie_{F}^\pi$ of sifted-colimit-preserving endofunctors of $\Mod_K$.}
 {We write  $C$ for the cofibre of $ L_{F/K}^\vee \rightarrow {T}$ in $\End(\Mod_K)$ and note that as a 
colimit of sifted-colimit-preserving functors, $C$ also  preserves sifted colimits. 
The universal property of $C$ gives   a natural transformation 
$C \rightarrow  \Lie_{F}^\pi.$
Restricting the cofibre sequence $ L_{F/K}^\vee \rightarrow {T}    \rightarrow C$ to $\Mod_{K,\leq 0}^{\ft}$ gives  a cofibre sequence of functors $\Mod_{K,\leq 0}^{\ft} \rightarrow \Mod_K$, and so the map $C \rightarrow  \Lie_{F}^\pi$ restricts to an equivalence on  $\Mod_{K,\leq 0}^{\ft}$. But by the equivalence in  \Cref{Plamonad}(3), this implies that the map $C \rightarrow  \Lie_{F}^\pi$  is an equivalence on all of $\Mod_K$.}
\end{proof}

Property $(1)$ of \Cref{maincons} in the introduction follows by construction, as $\LieAlgd^\pi_{F/K}$ preserves filtered colimits and geometric realisations. For property $(5)$, we prove:

 \begin{proposition}
On objects with vanishing anchor map, the functor $\LieAlgd^\pi_{F/K}$ \mbox{satisfies}
$$\LieAlgd^\pi_{F/K}\left(V\xrightarrow{0} L_{F/K}^\vee[1]\right) \simeq \left( \liep(V) \rightarrow L_{F/K}^\vee[1]\right).$$
\end{proposition}
\begin{proof} 
 {Recall   that we say that an 
anchor map vanishes  if it is nullhomotopic (in $\Mod_F$).} By the \vspace{-3pt}universal property of fibres,  we observe that $\fib: (\Mod_{F})_{/L_{F/K}^\vee[1]} \xrightarrow{} \Mod_F$ admits a left adjoint, which sends   a chain \vspace{-2pt}complex $V\in \Mod_F$ to the zero morphism  $(V \xrightarrow{0}L_{F/K}^\vee[1] )$. Writing $\LieAlgd^\pi_{F/K} (V\xrightarrow{0} L_{F/K}^\vee[1]) \simeq (X\rightarrow L_{F/K}^\vee[1] )$, we observe from \eqref{square}  that $X$ is the cofibre of the canonical morphism $L_{F/K}^\vee \rightarrow T(V)$, which is $\liep(V) $  {by \Cref{basiccofextended}.}
\end{proof}

\begin{remark} {
For $F=K$,  we have $L_{F/K} = 0$, and the resulting equivalence $\Mod_F \simeq (\Mod_F)_{/L_{F/K}^\vee[1]}$, $V \mapsto (V \xrightarrow{} 0)$
identifies the monads $\Lie_{F}^\pi $ on $\Mod_F$ and  $\LieAlgd^\pi_{F/K}$ on $(\Mod_F)_{L_{F/K}^\vee[1]}$.
Hence the theory of $F/K$-partition Lie algebroids reduces to the theory of partition Lie algebras for $F=K$.
} 
\end{remark}

\subsection{Koszul duality} \label{KDsec} As before, let $F/K$ be a finite purely inseparable field extension.
The main aim of this section is to construct an adjunction 
$$\mathfrak{D}: \SCR_{K//F} ^{\op} \rightleftarrows \LieAlgd^\pi_{F/K}: C^\ast $$
implementing Koszul duality, and show it restricts to an equivalence on suitably finite objects.\vspace{4pt}

First, we check that the conditions of \Cref{cNdef} are unchanged by base change along $K \subset F$.

\begin{lemma}\label{tensornoetherian}
 {Let $F/K$ be a finite extension of fields.  Then $B\in \SCR_{K}$ is Noetherian if and only if $R=F\otimes_KB\in \SCR_F$ is Noetherian.}
\end{lemma}
\begin{proof}
{First we reduce to the case where $B$ is discrete.  Since $K\to F$ is  flat, we have $\pi_i(F\otimes_K B)\cong F\otimes_K \pi_i(B)$.  As 
 faithfully flat maps are closed under base change and $K\to F$
 is faithfully flat,  we see that $\pi_0(B)  \rightarrow F\otimes_K \pi_0(B)$ is faithfully flat as well.  By \cite[03C4]{stacks-project},  this implies that $F\otimes_K \pi_i(B)$ is a finite $F\otimes_K \pi_0(B)$-module if and only if $\pi_i(B)$ is a finite $\pi_0(B)$-module.  So from now on we may assume that $B$ is discrete.}

Assume first that $B$ is discrete and Noetherian.  We can write $F = K[X_1,\ldots,X_n]/I$ for some suitably chosen ideal $I$, using that $F/K$ is finite. 
The Hilbert basis theorem therefore implies that  $F \otimes_K B$ receives a surjection from a Noetherian ring, and is therefore itself Noetherian. 
Conversely, for any ascending chain of ideals $I_0 \subset I_1 \subset \ldots$ in $B$, the chain $F\otimes_K I_0 \subset F\otimes_K I_1 \subset \ldots$ in 
$F \otimes_K B$ must stabilise. Since $K \hookrightarrow F$ is faithfully flat, this implies that the original chain must stabilise too, hence $B$ is Noetherian.
\end{proof}

\begin{corollary}\label{finite}
 {If $F/K$ is a finite extension of fields,  and $B\in \SCR_{K//F}^{}$ is Noetherian, then 
$$ F \otimes_B L_{B/K} \in \Mod_{F,\geq 0}^{\ft}$$
is connective and of finite type.}
\end{corollary}
\begin{proof}
We have an equivalence $F \otimes_B L_{B/K} \simeq \cot(F \otimes_K B) =   L_{F/(F \otimes_K B)}[-1]$. Since  $F \otimes_K B$ is a  Noetherian augmented $F$-algebra, \Cref{cnoethcot}  implies the result.
\end{proof}

\begin{proposition} \label{cNprop}
 {Let $F/K$ be a finite purely inseparable field extension.  An object $B\in \SCR_{K//F}$ is  {complete local Noetherian} if and only if the augmented $F$-algebra $R=F \otimes_K B \in  \SCR_{F//F}$ has this property.}
\end{proposition}
\begin{proof}
It follows from \Cref{tensornoetherian} that $B$ is Noetherian if and only if $R$ is Noetherian.  Therefore it remains to prove that $\pi_0(B)$ is a complete local (discrete) ring if and only if $\pi_0(R)\simeq F\otimes_K\pi_0(B)$ is a complete local ring.  Here we used the faithful flatness of $K\to F$.  Therefore we may assume that $B$ and $F\otimes_K B$ are both discrete. 

The main theorem of \cite{MR360568}  shows that $R$ is local if and only if $B$ and $F\otimes_K (B/\mathfrak{m})$ are both local, where $\mathfrak{m}$ is the maximal ideal of $B$.  So it remains to show that if $B$ is local then $F\otimes_K (B/\mathfrak{m})$ is local, which follows from \cite[Proposition]{MR360568} since $F/K$ is purely inseparable.

Finally we deal with completeness.  Consider the exact sequence 
$$0\to B\to \hat{B}\to M\to 0,$$ where $B$ is complete if and only if $M=0$.  Tensoring this sequence with $F$ we obtain
$$0\to R\to \varprojlim_k \left(R/\mathfrak{m}_B^kR\right)\to F\otimes_K M\to 0.$$
Note that $M=0$ if and only if $F\otimes_K M=0$ by faithful flatness of $F/K$.  This imples that $R$ is $\mathfrak{m}_B$-complete if and only if $B$ is $\mathfrak{m}_B$-complete.   Next, we claim that the two systems of ideals $(\mathfrak{m}_BR)^k$ and $\mathfrak{m}_{R}^k$ are cofinal. Then it will follow that $R$ is $\mathfrak{m}_B$-complete if and only if it is $\mathfrak{m}_{R}$-complete by \cite[0319]{stacks-project}, and we are done. 

To show the cofinality, let $e$ be the exponent of $F/K$.
Given $x = \sum_{i=1}^n f_{i} \otimes b_{i} \in \mathfrak{m}_R$, we can write 
$x^{p^e} $ as $ \sum_{i=1}^n f_{i}^{p^e} \otimes b_{i}^{p^e}$, which belongs to $ \mathfrak{m}_R^{p^e} \cap (K \otimes_K B)  \subset K \otimes_K \mathfrak{m}_B \subset  \mathfrak{m}_B R .$
Hence  $ \mathfrak{m}_R^{p^e} \subset  \mathfrak{m}_BR$, and more generally  
$(\mathfrak{m}_R^k)^{p^e} \subset  (\mathfrak{m}_BR)^k $.
Conversely, it is clear that $(\mathfrak{m}_BR)^k \subset \mathfrak{m}_R^k$ for all $k$ since $\mathfrak{m}_R$ is maximal.

\end{proof}

 {
To construct the Koszul duality functor on simplicial commutative $K$-algebras over $F$,  we will follow the 
 strategy of \mbox{\cite[Construction 4.50]{brantner2019deformation}}.   First, let us record a simple  \mbox{categorical observation:}
\begin{proposition}
Let $c$ be an object in a compactly generated $\infty$-category $\mathcal{C}$.
Then an object  $(x\rightarrow c) \in \mathcal{C}_{/c}$ is compact if and only if $x\in \mathcal{C}$ is compact. Moreover, $\mathcal{C}_{/c}$ is compactly generated.
\end{proposition}}
\begin{proof} {
First,   observe that if $D: I \rightarrow \mathcal{C}_{/c}$, $i \mapsto (y_i \rightarrow c)$ is  a   diagram in $\mathcal{C}_{/c}$, the colimit of $D$ is given by the canonical morphism $( \colim_{I} y_i \rightarrow c)$ out  of the colimit of the \mbox{diagram $\tilde{D}: I \rightarrow \mathcal{C}$, $i \mapsto y_i$ in $\mathcal{C}$.}}

 {Now assume that we are given  $(x\xrightarrow{\epsilon } c) \in \mathcal{C}_{/c}$ with $x\in \mathcal{C}$ is compact. Given any other object $(y\rightarrow c) \in \mathcal{C}_{/c}$,  there is an equivalence
$\Map_{\mathcal{C}_{/c}}(x\rightarrow c,  y\rightarrow c )  \simeq \fib_{\epsilon} \left(\Map_{\mathcal{C}}(x,y) \rightarrow \Map_{\mathcal{C}}(x,c) \right).$
As filtered colimits commute with finite limits in spaces, this implies that $(x\rightarrow c) $ is  compact.}

 {Conversely, assume that $(x\xrightarrow{  } c) \in \mathcal{C}_{/c}$ is compact. Given a filtered diagram $D: I  \rightarrow \mathcal{C}, \  i \mapsto y_i$, we define a filtered diagram $\widetilde{D}:  I  \rightarrow \mathcal{C}_{/c},   \ i \mapsto (y_i \times c \rightarrow c)$ and obtain the  commutative square
 $$\xymatrix{
 \colim_I \Map_{\mathcal{C}_{/c}}(x\rightarrow c,  \ y_i \times c \rightarrow c) \ar[d]^{\simeq}  \ar[r]^{\ \ \  \ \ \ \ \ \ \   \simeq} & \colim_I   \Map_{\mathcal{C}}(x ,   y_i )  \ar[d] \\ 
\Map_{\mathcal{C}_{/c}}(x\rightarrow c,  \ \colim_I  (y_i \times c \rightarrow c))\ar[r]^{\ \ \  \ \ \ \ \ \ \   \simeq} &  \Map_{\mathcal{C}}(x ,   \colim_I y_i ).}$$
To see that the lower horizontal map is an equivalence, we use that
filtered colimits commute with finite limits
 in compactly generated $\infty$-categories.  As $D$ was arbitrary,  we see that $x\in \mathcal{C}$ is compact.}\vspace{2pt}

 {To see that $\mathcal{C}_{/c}$ is compactly generated,  let us fix an object $(x\rightarrow c)$ and
write $\mathcal{D}_0 \subset  \mathcal{C}_{/c}$ and $\widetilde{\mathcal{D}}_0 \subset  (\mathcal{C}_{/c})_{/(x\rightarrow c)}$
  for the full subcategories of compact objects  in $\mathcal{C}$ and $\mathcal{C}_{/c}$
 mapping to $x$ and $(x\rightarrow c)$,  respectively.
We note that $\colim_{(y\rightarrow c) \in  \widetilde{\mathcal{D}}_0} (y \rightarrow c) \simeq (
(\colim_{y \in   {\mathcal{D}}_0} y) \rightarrow c) \simeq (x \rightarrow c)$.}
\end{proof}
We may therefore deduce:

\begin{corollary}\label{corcompact} {
The $\infty$-category $\SCR_{K//F}$ is compactly generated.  An object $B\in \SCR_{K//F}$ is compact if and only if, as a $K$-algebra,  $B$ is a retract of a finitely presented $K$-algebra.}
\end{corollary}

\begin{notation} {
Write $ \SCR_{K//F}^{\wafp} \subset \SCR_{K//F}$ for the full subcategory spanned by all $B$ with $$\cot_{K//F}(B) \in \Mod_{F,\geq 0}^{\ft}.$$
Combining \Cref{corcompact} and \cite[Proposition 3.2.14]{lurie2004derived}, we see that $\SCR_{K//F}^{\wafp}$ contains all compact objects of $ \SCR_{K//F}$, and by \Cref{finite}, it also contains all Noetherian objects.}
\end{notation}

\begin{construction}[Koszul duality adjunction]\label{constructadjunction}\ 

First,   note that by  construction of the monad $\myovline{T}{1.0pt} $ in 
 \Cref{maindef} (1), there is  {a canonical functor $$\SCR_{K//F}^{\op}\rightarrow \Alg_{\myovline{T}{1.0pt}}\ ,\    B \mapsto  \cot_{K//F}(B)^\vee.  $$ 
 Taking the opposite gives a functor $\SCR_{K//F} \rightarrow \Alg_{\myovline{T}{1.0pt} } ^{\op}.$}
  
  {
As the above functor sends objects in $ \SCR_{K//F}^{\wafp}$ to  modules in $\Mod_{F, \leq 0}^{\ft}$,  
and since the monads  $\myovline{T}{1.0pt} $ and $T$ are canonically equivalent  on $\Mod_{F, \leq 0}^{\ft}$, restriction   gives rise to a functor  
    $$\mathfrak{D}: \SCR_{K//F}^{\wafp} \rightarrow \Alg_{T} ^{\op} \simeq \Alg_{\LieAlgd^\pi_{F/K}} ^{\op}.$$  
 {Since
$\SCR_{K//F}$ is compactly generated (cf.\  \Cref{corcompact})
and 
 $\cot_{K//F}$ preserves filtered colimits,  the functor 
$\cot_{K//F}$ and its restriction to $ \SCR_{K//F}^{\wafp}$ are both  left Kan extended from  \mbox{compact objects.}}
As $T$ preserves filtered colimits, this implies that the functor $\mathfrak{D}$  above is valso left Kan extended from these compact objects.}
 {We can therefore
left Kan extend further to all of $\SCR_{K//F}$ to obtain a colimit-preserving functor  $$\mathfrak{D}: \SCR_{K//F}  \rightarrow   \Alg_{\LieAlgd^\pi_{F/K}} ^{\op}.$$
Its right adjoint is called the \emph{Chevalley--Eilenberg complex} and written as  $$C^\ast:  \Alg_{\LieAlgd^\pi_{F/K}} ^{\op}  \rightarrow \SCR_{K//F}.$$}
\end{construction}
\ 

Koszul duality does not lose any information on  complete local Noetherian $K$-algebras over $F$:
\begin{theorem} \label{AffineKD}
The adjunction $(\mathfrak{D}\dashv C^\ast)$ from \Cref{constructadjunction} restricts to a contravariant  \mbox{equivalence}
$$\SCR_{K//F}^{\cN} \simeq  (\mathcal{D}_0) ^{\op}$$
between the $\infty$-category of complete local Noetherian objects in $\SCR_{K//F}$ and  the full subcategory 
$$\mathcal{D}_0 \subset \Alg_{\LieAlgd^\pi_{F/K}} $$
consisting of all  $(\mathfrak{g} \xrightarrow{\rho}   L_{F/K}^\vee[1])$ 
for which   $\fib(\rho)\in \Mod_{F,\leq 0}^{\ft}$ is coconnective and \mbox{of finite type.}
\end{theorem} 
\begin{proof} 
First, we prove that the adjunction   $$\cot_{K//F}:\SCR_{K//F}^{\cN} \leftrightarrows \Mod_{F,\geq 0}^{\ft}: \sqz_{K//F} $$ is comonadic  by verifying  the conditions of the Barr--Beck--Lurie theorem \cite[Theorem 4.7.3.5]{lurie2014higher};
 {note that the associated comonad is related to the monad  $\myovline{T}{1.0pt}$ on $\Mod_{F,\leq 0}^{\ft}$ (cf.\ \Cref{maindef})
via linear duality.}
To this end, assume that $B^\bullet  \in \SCR_{K//F}^{\cN}$ is a cosimplicial diagram for which $$\cot_{K//F}(B^\bullet) \simeq \cot_F(F \otimes_K B^\bullet) \in \Mod_{F, \geq 0}^{\ft}$$ admits a splitting. Applying \cite[Theorem 4.20.(1)]{brantner2019deformation}  to the setup of $\SCR_{F}^{\aug}$ (as explained in \cite[Section $5.2$]{brantner2019deformation}), we see that the limit  $\Tot(F \otimes_K B^\bullet) \in \SCR_F^{\aug}$ exists, is computed in $\Mod_F$, and belongs to $\SCR_F^{\cN}$. Moreover, we obtain  a canonical   equivalence
$$\cot_F(\Tot(F\otimes_{K} B^\bullet )) \xrightarrow{\ \simeq \ } \Tot( \cot_F(F \otimes_K B^\bullet) ).$$

The map $\label{equiv} F \otimes_K \Tot(B^\bullet)  \rightarrow \Tot(F \otimes_K B^\bullet) $
is an equivalence because  {the totalisation preserves finite colimits and $F$ is a finite free $K$-module.} 
 We deduce from \Cref{cNprop} that $ \Tot(B^\bullet)\in \SCR_{K//F}^{\cN}$ is complete local Noetherian,   
 computed in $\Mod_K$, and satisfies $$\cot_F(F\otimes_{K} \Tot(B^\bullet )) \xrightarrow{\ \simeq \ } \Tot( \cot_F(F \otimes_K B^\bullet) ).$$

Finally, the functor  $\cot_F(F\otimes_{K} (-)): \SCR_{K//F}^{\cN} \rightarrow \Mod_F $ is conservative 
as this is evidently true for  $F \otimes_K (-): \SCR_{K//F}^{\cN} \rightarrow \SCR_{F//F}^{\cN} $  and also holds for 
$\cot_F: \SCR_{F//F}^{\cN} \rightarrow \Mod_F$ by the proof of \cite[Theorem 4.20]{brantner2019deformation} applied to the setting in Section 5.2 of [op.cit.].   {Note that this is where completeness is crucially used.}  \vspace{3pt}

To finish the proof,  {let us observe the following triangle of adjunctions:
$$\xymatrix@+2pc{
\SCR_{K//F} \ar@<.0ex>[rd]_{\sqz_{K//F}(-^\vee)\ \ \ \ } \ar@<.5ex>[r]^{\mathfrak{D} \  }   &\ar@<.5ex>[l]^{C^\ast}  \Alg_{\LieAlgd^\pi_{F/K}} ^{\op}  \ar@<2ex>[d]_{\Free\ } \\ 
 &  \ar@<1ex>[lu]_{ \ \  \ \cot_{K//F}^\vee\ \ \ \ }  \ar@<-1.0ex>[u]_{\ \Forget} \Mod_F^{\op}}$$
For   $V\in \Mod_{F,\leq 0}^{\ft} $, the counit of the adjunction $\mathfrak{D} \dashv C^\ast$ gives a map $ \Free(V) \rightarrow \mathfrak{D}(C^\ast(\Free(V))) $  in $ \Alg_{\LieAlgd^\pi_{F/K}} $, which 
is an equivalence because applying the conservative forgetful functor   gives the equivalence 
$\Forget( \Free(V)) = T(V) \xrightarrow{\simeq} \myovline{T}{1.0pt}(V) = \cot_{K//F}^\vee( \sqz_{K//F}^\vee(V)). $
Since the composite $ C^\ast(\Free(V)) \rightarrow C^\ast( \mathfrak{D}(C^\ast(\Free(V))))  \rightarrow  C^\ast(\Free(V))$ of unit and counit is always an equivalence,  and we may therefore deduce that  the unit 
$$  B  \rightarrow
C^\ast(\mathfrak{D}(B))$$ is an  equivalence for all $B = \sqz_{K//F}(V^\vee)$ with $V\in \Mod_{F,\leq 0}^{\ft}.$} 

For a general $B\in \SCR_{K//F}^{\cN}$, we combine   the comonadicity established above with the  equivalence
 \mbox{$(-)^\vee: \Mod^{\ft}_{F, \leq 0} \simeq (\Mod^{\ft}_{F, \geq 0}) ^{\op}$}
to write  $B \simeq \Tot(B^\bullet)$ as a totalisation preserved by $\mathfrak{D}$ of a diagram of trivial square-zero extensions. 
Here, we use the  {cobar} resolution coming from the  {comonadic}  adjunction $\cot_{K//F}:\SCR_{K//F}^{\cN} \leftrightarrows \Mod_{F,\geq 0}^{\ft}: \sqz_{K//F} $, cf.\ \cite[Proposition 4.7.3.3]{lurie2014higher}.\vspace{2pt}

We may therefore deduce that   the unit $$\eta_B: B\rightarrow C^\ast (\mathfrak{D}(B))$$ is an equivalence  for all $B\in \SCR_{K//F}^{\cN}$, which in turn implies that $\mathfrak{D}|_{\SCR_{K//F}^{\cN}}$ is fully faithful.
 {Indeed, this last implication follows from a well-known categorical argument which we  recall for the reader's convenience.
Given
$B_1, B_2 \in \SCR_{K//F}^{\cN}$,  naturality of the unit shows that the composite  
\begin{equation*} \label{eq1}
\begin{split}
\Map_{\SCR_{K//F}}(B_1, B_2)& \xrightarrow{ \ \      \mathfrak{D}(-)  \ } \Map_{\Alg_{\LieAlgd^\pi_{F/K}}}\left(\mathfrak{D}(B_2), \mathfrak{D}(B_1)\right) \\
 & \xrightarrow{ \ \    C^\ast(-) \ } \Map_{\SCR_{K//F}}\left(C^\ast(\mathfrak{D}(B_1)),C^\ast(\mathfrak{D}(B_2))\right) \\
 & \xrightarrow{\ - \ \circ \ \eta_{B_1}} \Map_{\SCR_{K//F}}\left( B_1,  C^\ast(\mathfrak{D}(B_2))\right) 
\end{split}
\end{equation*} 
is given by postcomposition with the unit $\eta_{B_2}: B_2\rightarrow C^\ast (\mathfrak{D}(B_2))$ and therefore an equivalence.
As   the second two arrows compose to an equivalence by the defining property of adjunctions (cf.\ \cite[Section 5.2.2]{lurie2009higher})),  the first map is an equivalence and so  $\mathfrak{D}|_{\SCR_{K//F}^{\cN}}$ is fully faithful.
}

To identify the essential image of $\mathfrak{D}|_{\SCR_{K//F}^{\cN}}$, we unravel the definitions to factor this functor as a 
chain of equivalences,   once more using the comonadicity result established above:
$$\SCR_{K//F}^{\cN} \simeq \coAlg_{\cot_{K//F}\circ \sqz_{K//F}}(\Mod_{F,\geq 0}^{\ft})
\simeq \Alg_T(\Mod_{F,\leq 0}^{\ft}) \simeq \mathcal{D}_0.$$
\end{proof}

\subsection{An interlude on hypercoverings}  

 {The aim of this section and the next is to relate our partition Lie algebroids with parition Lie algebras and formal moduli problems via two natural functors described in  \Cref{maincons} (3) and (4).  Since this requires more knowledge of derived algebraic geometry than we have used thus far, the reader interested in the Galois correspondence may wish to skip to section 4 in which we finish the proof of \Cref{fundamental_theorem}.}

To construct functors on partition Lie \mbox{algebroids,} we will use the theory of hypercoverings, which allows us to proceed in two steps. First, we  define a functor on certain small Lie algebroids which admit \mbox{an interpretation in terms of rings,}
second, we extend it via distinguished  simplicial  resolutions known as hypercoverings.

The theory of hypercoverings originated in the work of Verdier   \cite{artin1972theorie}, and have since   been revisited by many authors, most recently also in a   higher categorical setting (cf.\ e.g.\ \cite{dugger2004hypercovers},  \cite[Section 3.2]{toen2005homotopical}, \cite[Section 6.5.3]{lurie2009higher} \cite[Section  7.2.1]{lurie2014higher}, or \cite[Appendix]{brantner2019deformation}).

Rather than delving into the general theory, we shall only discuss  hypercoverings  in the 
context of partition Lie algebroids, relying on the  \vspace{2pt} more general results established  in \cite[Appendix]{brantner2019deformation}.

 Given a finite purely inseparable field extension $F/K$, we will use two distinguished classes of objects to build general Lie algebroids. \vspace{-3pt}

\begin{notation} \label{smallalgebroids}
Write $\mathcal{C}_0 \subset \mathcal{C}_1 \subset  \mathcal{C}:=\Alg_{\LieAlgd^\pi_{F/K} }$ 
for the full subcategories of free algebras
$$\LieAlgd^\pi_{F/K}\left(V\xrightarrow{0} L_{F/K}^\vee[1]\right), $$ 
where $V \in \Mod_F$ is  assumed to be perfect coconnective for $ \mathcal{C}_0$ and  coconnective for $ \mathcal{C}_1$, respectively.
\end{notation}
We introduce the class of morphisms which one can think of as some kind of  covers:
\begin{definition}[Coconnective anchor surjections] A map  of $F/K$-partition Lie algebroids 
$$\xymatrix{
\mathfrak{g}_1  \ar[r] \ar[rd]_{\rho_1} &  \mathfrak{g}_2 \ar[d]^{\rho_2} \\ 
& L_{F/K}^\vee[1]}$$

is a \textit{coconnective anchor surjection} if it induces a surjection $\pi_i(\fib(\rho_1)) \rightarrow \pi_i(\fib(\rho_2))$ \mbox{for all $i \leq 0$.}
\end{definition}

We define a  distinguished family of simplicial resolutions:
\begin{definition}[Hypercoverings]\label{hyperhyper}
Let $Z = (\mathfrak{g} \rightarrow L_{F/K}^\vee[1])$ be an $F/K$-partition Lie algebroid. \mbox{A \textit{hypercovering} for} $Z$ is an augmented simplicial $F/K$-partition Lie algebroid $$X_\bullet \rightarrow Z $$
 satisfying the following conditions:
\begin{enumerate}
\item Each object  $X_n$ belongs to  $ \mathcal{C}_1$;
\item Each map $X_i \rightarrow Z$ is a coconnective anchor surjection;
\item All matching objects $M_n(X_\bullet)$ and all latching objects $L_n(X_\bullet)$ exist in $(\mathcal{C}_{/Z})_{}'$, the $\infty$-category
of all $F/K$-partition Lie algebroids mapping to $Z$ via an anchor surjection  (cf.\ \cite[Definition 8.1]{brantner2019deformation} for a definition of matching and latching objects).
\item Each natural map $X_n \rightarrow M_n(X_\bullet)$ is a coconnective anchor surjection. Each natural map $L_n(X_\bullet) \rightarrow X_n$ expresses $X_n$ as a coproduct of $L_n(X_\bullet)$ with an object in the subcategory $(\mathcal{C}_{1/Z})' 
\subset (\mathcal{C}_{/Z})_{}'$, spanned by all $X \rightarrow Z$ with $X\in \mathcal{C}_1$.
\end{enumerate}
\end{definition} 
Such hypercoverings exist in abundance: 
\begin{proposition}[Hypercoverings  and left  Kan extensions] \label{hypkan}
Any    $Z\in \Alg_{\LieAlgd^\pi_{F/K}}$ admits a hypercovering $X_\bullet \rightarrow Z$, which is in fact  a colimit diagram.
If $\mathcal{D}$ is a  presentable $\infty$-category and   $F: \Alg_{\LieAlgd^\pi_{F/K}} \rightarrow \mathcal{D}$ is left Kan extended from $\mathcal{C}_0$, we have a natural equivalence 
$|F(X_\bullet)| \simeq F(Z). $
Moreover, any sifted-colimit-preserving functor $\Alg_{\LieAlgd^\pi_{F/K}} \rightarrow \mathcal{D}$  is left Kan extended from $\mathcal{C}_0$.
\end{proposition}
\begin{proof}
Recall the   sifted-colimit-preserving monad $T$ on $\Mod_F$ constructed in 
\Cref{maindef} (2).  {Starting with the  set 
 $\mathcal{F}_0$ of all $T$-algebras of the form  $T(V)$ with $V\in \Perf_{F, \leq 0}$ perfect coconnective, \cite[Construction 8.9]{brantner2019deformation} gives a weakly orthogonal pair $(\mathcal{F}_1, \mathscr{S})$ in the sense of  Definition 8.4 in [op.cit.].   
Here $ \mathcal{F}_1$ is the class of objects which are coproducts of objects in  $ \mathcal{F}_0$; hence $\mathcal{F}_1$ consists of all $T$-algebras  of the form $T(V)$ with $V\in \Mod_{F, \leq 0}$.
The class of morphisms $\mathscr{S}$ consists of all $f: X_1\rightarrow X_2$ in $\Alg_T$
for which the induced map $\pi_0\Map_{\Alg_T}(F,X_1) \rightarrow \pi_0 \Map_{\Alg_T}(F,X_2)$ is surjective for all $F\in \mathcal{F}_0 $. Unravelling the definitions, we see that $\mathscr{S}$ is
given by the family of morphisms of $T$-algebras $X_1 \rightarrow X_2$ which induce surjections on $\pi_i$ for all $i\leq 0$. }

By  Lemma 8.6 in [op.cit.], any $B\in \Alg_T$ admits an $(\mathcal{F}_1,\mathscr{S})$-hypercovering $A_\bullet \rightarrow B$ in the sense of Definition 8.7 in [op.cit.]. We explain in Example 8.12 in [op.cit.] that $A_\bullet \rightarrow B$ is a  colimit diagram, that any  $G: \Alg_T \rightarrow \mathcal{D}$  which is left Kan extended from $\mathcal{F}_0$ satisfies $|G(A_\bullet)| \simeq G(B)$, and    that any sifted-colimit-preserving functor $F: \Alg_T \rightarrow \mathcal{D}$  is left \mbox{Kan extended from $\mathcal{F}_0$.}
 
Finally, we recall the \vspace{-3pt} equivalence 
$\Alg_{\LieAlgd^\pi_{F/K}} \simeq \Alg_T(\Mod_F)$ from \Cref{maindef} (3),  lifting the fibre functor $(\mathfrak{g} \xrightarrow{\rho} L_{F/K}^\vee[1]) \mapsto \fib(\rho)$. Unraveling the definitions, we use diagram \eqref{square} in the proof of \Cref{maindef} to observe that $\mathcal{C}_0 \subset  \mathcal{C}_1$ and $S$
correspond to 
$\mathcal{F}_0 \subset  \mathcal{F}_1$ and $\mathscr{S}$  under this equivalence, which implies the various  assertions.
\end{proof}

We can now deduce a convenient extension  result:
\begin{proposition}\label{functor_extensions}
Let $U: \mathcal{D} \rightarrow \mathcal{E}$ be a functor of presentable $\infty$-categories creating sifted colimits.
Let $\mathcal{C}_0\subset \mathcal{C}$ be as in \Cref{smallalgebroids},  and assume we are given a functor $G_0: \mathcal{C}_0 \rightarrow  \mathcal{D}$    for which the composite  $U \circ G_0: \mathcal{C}_0 \rightarrow   \mathcal{E}$ admits a sifted-colimit-preserving extension   \mbox{$ {H}: \mathcal{C} \rightarrow   \mathcal{E}$.} 
Then $G_0$ admits  a unique sifted-colimit-preserving extension  $ G: \mathcal{C} \rightarrow \mathcal{D}$  filling the following diagram:
$$\xymatrix{
 \mathcal{C}_0   \ar[r]^{G_0} \ar@{^{(}->}[d] &  \mathcal{D} \ar[d]^{U} \\ 
\mathcal{C}  \ar[r]_{ {H}}\ar@{-->}[ru]^G& \mathcal{E}.}$$

\end{proposition}
\begin{proof}
The left Kan extension $G := \Lan_{\mathcal{C}_0}^{\mathcal{C}}(G_0)$  
   to $\mathcal{C}$  exists by \cite[4.3.2.14]{lurie2009higher} \mbox{as  $\mathcal{D}$ is presentable.}  {We will now verify that  $G$  satisfies the conclusion in the proposition, thereby proving 
 existence.}

As $ {H}: \mathcal{C} \rightarrow \mathcal{E}$ preserves sifted colimits, it is left Kan extended from $\mathcal{C}_0$ by \Cref{hypkan}, and the equivalence  $ {H}|_{\mathcal{C}_0} \simeq U \circ G_0$ extends to a natural transformation 
\mbox{$\alpha:  {H} \rightarrow U \circ G$.}
As $G$ is left Kan extended from its values on compact objects, it preserves filtered colimits, and the same holds true for $U$ and $ {H}$. Since any object in $\mathcal{C}_1$ is a filtered colimit of objects in $\mathcal{C}_0$, we deduce that 
$\alpha$ is an equivalence on $\mathcal{C}_1$.
Given a general  $Z\in \mathcal{C}$. We use \Cref{functor_extensions} to
pick a hypercovering $X_\bullet \rightarrow Z$
, which  is also a colimit diagram. We then consider the following commuting square:

$$\xymatrix{
| {H}(X_\bullet)|  \ar[r] \ar[d] &   {H}(Z) \ar[d] \\ 
|(U\circ G)(X_\bullet)|  \ar[r]& (U\circ G)(Z)}$$

The top horizontal map is an equivalence since  $ {H}$ preserves geometric realisations, 
the bottom  since $U$ preserves realisations and $|G(X_\bullet)| \simeq G(Z)$ by \Cref{hypkan}, and the left vertical  map since $\alpha$ is an equivalence on all objects in $\mathcal{C}_1$ by our previous considerations. 
Hence $\alpha:  {H}  \xrightarrow{\simeq} U \circ G$ is an equivalence, which implies that $G$ preserves sifted colimits as $ {H}$ preserves  and $U$ creates them. 

The functor $G$ is unique as any sifted-colimit-preserving functor $\mathcal{C} \rightarrow \mathcal{D}$ must be left Kan extended from $\mathcal{C}_0$ by \Cref{hypkan}.\vspace{-2pt} 
\end{proof}

\subsection{Functors on partition Lie algebroids}
 {The main goal of this section is to construct two  functors  
$ \Alg_{\LieAlgd^\pi_{F/K}} \rightarrow (\Alg_{\Lie^\pi_{F}})_{L_{F/K}^\vee/}, $
and $\Alg_{\LieAlgd^\pi_{F/K}} \rightarrow (\Alg_{\Lie^\pi_{K}})_{/L_{F/K}^\vee[1]}$
on partition Lie algebroids which enhance  the fiber functor and forgetful functor as described in \Cref{maincons} (4) and (3), respectively. This illustrates} that  partition Lie algebroids really do behave like a derived version of classical Lie algebroids. 
We start  with the fibre functor described in \mbox{\Cref{maincons} $(4)$:}

\begin{proposition}[Fibre of the anchor]\label{fibre_anchor}
Consider the functor 
 \mbox{$ \fib: \Alg_{\LieAlgd^\pi_{F/K}}\hspace{-9pt} \rightarrow (\Mod_F)_{L_{F/K}^\vee/}$} sending $(\mathfrak{g} \xrightarrow{\rho} L_{F/K}^\vee[1])$ to $\fib(\rho)$. Then $\fib$ admits a  {canonical} sifted-colimit-preserving lift
	$$ \Alg_{\LieAlgd^\pi_{F/K}} \rightarrow (\Alg_{\Lie^\pi_{F}})_{L_{F/K}^\vee/}, $$
	where $L_{F/K}^\vee$  is the $F$-partition Lie \vspace{3pt} algebra of \Cref{firstex} (1)
\end{proposition}
\begin{proof}
 By \Cref{AffineKD}, the assignment $R \mapsto \mathfrak{D}(R) = (L_{F/R}^\vee[1] \rightarrow L_{F/K}^\vee[1]) \in \Alg_{\LieAlgd^\pi_{F/K}} $ 
induces a canonical equivalence 
\mbox{$(\SCR_{K//F}^{\cN}) ^{\op}  \xrightarrow{\simeq}\mathcal{D}_0$} with inverse
$\mathfrak{g} \mapsto C^\ast(\mathfrak{g})$. Here
 $\mathcal{D}_0$ consists of all $F/K$-partition Lie algebroids 
   $(\mathfrak{g} \xrightarrow{\rho} L_{F/K}^\vee [1])\in \Alg_{\LieAlgd^\pi_{F/K}} $ which satisfy \vspace{-2pt} \mbox{$\fib(\rho) \in \Mod_{F, \leq 0}^{\ft}$.}
Note that $\mathcal{D}_0$ contains the full subcategory   $\mathcal{C}_0 \subset \mathcal{D}_0 $ of free algebroids on objects $(V \xrightarrow{0} L_{F/K}[1])$ with \vspace{3pt} $V \in \Perf_{F,\leq 0}$, cf.\ \Cref{smallalgebroids}. 
Let us now consider the following composite functor $G_0$:\vspace{-5pt}
$$\mathcal{C}_0 \hookrightarrow  \mathcal{D}_0 \xrightarrow{\simeq}  (\SCR_{K//F}^{\cN})^{\mathrm{op}} \xrightarrow{B \mapsto  F \otimes_K B} \left((\SCR_{F//F}^{\cN})_{/F\otimes_K F}\right)^{\mathrm{op}} \xrightarrow{R \mapsto  ( L_{F/F\otimes_{_K} F}^\vee[1] \rightarrow L_{F/R}^\vee[1])}    (\mathrm{Alg}_{\liep})_{L_{F/K}^\vee/ },$$
where the final map is defined using the equivalence contained in \vspace{3pt} \cite[Corollary 5.46]{brantner2019deformation}.

 {
Write $\Forget: (\mathrm{Alg}_{\liep})_{L_{F/K}^\vee/ }\rightarrow \Mod_F,  (L_{F/K}^\vee \rightarrow M) \mapsto M$ for the forgetful functor.
Using 
\Cref{threemaps}, we  observe that the following diagram commutes up to homotopy:\vspace{10pt}
$$\xymatrix@+2pc{ 
  \mathcal{C}_0 \subset \mathcal{D}_0 \ar[rd]_{\id}  \ar[r]^{C^\ast}_{\simeq}& (\SCR_{K//F}^{\cN})^{\mathrm{op}} \ar[rrd]_[@!-14.65]*+<-4.3em>{  _{ \ \ \ \ \   \ \ \ \ \  \  \ \ \ \ \ \  \ \ \ \ \ \  \ \ \ \ \ B \mapsto \cot_{K//F}(B)^\vee \simeq \cot_F(F\otimes_K B)^\vee \simeq ( F\otimes_B L_{B/K})^\vee}\  \ \ }     \ar@<+1ex>[rd]_[@!-28.2]*+<-3.5em>{  _{ \  \ \ \ \ \  \  \ \ \ \ \ \  \ \ \ \ \ \  \ \ \ \ \ B \mapsto (L_{F/B}^\vee[1] \rightarrow L_{F/K}^\vee[1] )}\  \ \ }     \ar[r]^{  B \mapsto  F \otimes_K B \ \ \ \ \ }  \ar[d]_{\mathfrak{D}}^{\simeq }  &\left( (\SCR_{F//F}^{\cN})_{/F\otimes_K F} \right) ^{\mathrm{op}}\ar[r]^{ }\ar@<+1ex>[rd]^[@!-27]*+<-2em>{  _{ R \mapsto \cot_F(R)^{\vee} = L_{F/R}^\vee[1]}\  \ \ }  &  (\mathrm{Alg}_{\liep})_{L_{F/K}^\vee/ }\ar[d]^{\Forget} \\ 
 &  \mathcal{D}_0\ar[r]&   (\Mod_F)_{/L_{F/K}^\vee[1]} \ar[r]_{\fib}  & \Mod_{F}.}\vspace{10pt}$$
 Hence  $\Forget\circ G_0$ simply takes the fibre of the anchor map, and therefore extends to a sifted-colimit-preserving functor   $ {H} : \Alg_{\LieAlgd^\pi_{F/K}} \rightarrow \Mod_F$ sending $(\mathfrak{g} \xrightarrow{\rho} L_{F/K}^\vee[1])$ to $\fib(\rho)$. 
  \Cref{functor_extensions} then gives the desired sifted-colimit-preserving lift.  }
\end{proof}

To construct the  forgetful functor $U: \Alg_{\LieAlgd^\pi_{F/K}} \rightarrow (\Alg_{\Lie^\pi_{K}})_{/L_{F/K}^\vee[1]}$  {appearing}
 {in \Cref{maincons} $(3)$}, we will   consider 
 Kodaira--Spencer type formal moduli problems for $K$-schemes. 
To this end, we use the setup  of \cite[19.4]{lurie2016spectral} in a derived (rather than spectral) setting. 

\mbox{More precisely,  let}
$$\mathrm{Var}^+_\simeq:\SCR\to  {\mathcal{S}}$$  
denote the functor sending $B\in \SCR$ to the underlying Kan complex $\mathrm{Var}^+_\simeq(B)$ of $ \mathrm{Var}^+(B)$, the 
(essentially small) $\infty$-category of  maps of derived Deligne-Mumford stacks
$$Z\to \Spec(B)$$  which are proper, flat, and locally of almost finite presentation.  
 {Here a morphism 
$B_1 \rightarrow B_2$ is sent to the functor $\mathrm{Var}^+(B_1) \rightarrow \mathrm{Var}^+(B_2)$ 
obtained by pulling back a given $(Z\to \Spec(B_1)) \in \mathrm{Var}^+(B_1) $  along the morphism  $\Spec(B_2) \rightarrow \Spec(B_1)$.}

Given $B\in \SCR$ and a point $\eta \in \mathrm{Var}^+_\simeq(B)$ corresponding to  $Z\to \Spec(B)$, we may encode derived deformations of $Z$ by  the functor 
\begin{equation} \label{bigdef}
\Def^+_\eta : \SCR_{/B} \rightarrow  {\mathcal{S}}, \  A \mapsto  \mathrm{Var}^+_\simeq(A) \times_{\mathrm{Var}^+_\simeq(B)}\{\eta\}.\end{equation}
 {This functor is cohesive  (cf.\ \cite[Definition 3.4.1]{lurie2004derived})
 as this holds  for the functor 
 $\mathrm{Var}^+_\simeq$  by (the derived version of)    {\cite[Theorem 19.4.0.2]{lurie2016spectral}}. 
Note that the forgetful functor 
$\SCR_{/B} \rightarrow \SCR$   creates pullbacks. }

Any finite purely inseparable field extension $K\subset F$ gives a 
map of schemes  $\Spec(F)\to\Spec(K)$  {which is 
proper,
flat,
and locally of almost finite presentation, and 
\mbox{hence gives   a  point $\eta \in \mathrm{Var}^+_\simeq(K)$. }}
\ 

Restricting $\Def^+_\eta$ to  the $\infty$-category $\SCR_{K}^{\art}$ from \Cref{artinian}, we obtain a functor
$$\Def_{F/K}:\SCR_{K}^{\art} \rightarrow \mathcal{S}, \ \ A \mapsto  \Def_{F/K}(A):= \mathrm{Var}^+_\simeq(A) \times_{\mathrm{Var}^+_\simeq(K)}\{\eta\},  $$
which is  {the} formal moduli problem (over $K$)  encoding infinitesimal deformations of the $K$-scheme $\Spec(F)$, cf.\  \cite[Remark 19.4.4.1]{lurie2016spectral}.
The cited  remark also  proves the well-known fact that the tangent fibre of $\Def_{F/K}$ is given by $L_{F/K}^\vee [1]$. By \cite{brantner2019deformation}, Theorem 1.11], this tangent fibre is moreover equipped with the structure of a  partition Lie algebra  controlling the {formal moduli problem $\Def_{F/K}$,} as 
was already asserted in  \Cref{firstex} $(2)$.
With this in hand, we can show: 
\begin{proposition}[Forgetful functor]
	The forgetful\vspace{1pt}  functor $\Alg_{\LieAlgd^\pi_{F/K}}  \rightarrow (\Mod_K)_{/L_{F/K}^\vee[1]}$
sending an algebroid $(\mathfrak{g} \rightarrow L_{F/K}^\vee[1])$ to its underlying object in   $(\Mod_K)_{/L_{F/K}^\vee[1]}$ 
	lifts  {canonically} to a sifted-colimit-preserving functor 
	$$U: \Alg_{\LieAlgd^\pi_{F/K}} \rightarrow (\Alg_{\Lie^\pi_{K}})_{/L_{F/K}^\vee[1]},$$
	where $L_{F/K}^\vee[1]$  is the $K$-partition Lie algebra of \Cref{firstex} (2).
\end{proposition}

\begin{proof}[Proof]  
We will again use  the equivalence in \Cref{plieeq}. First, \mbox{we will construct  a functor} $$\Def_{F/\bullet/K}:(\SCR_{K//F}^{\cN}) ^{\op}\to (\mathrm{Moduli}_{K})_{/\Def_{F/K}}$$
which sends a given $B\in \SCR_{K//F}^{\cN}$  {with maximal ideal $\mathfrak{m}$}
to  {the formal moduli problem   (over $K$)  encoding}  {compatible families of deformations of the morphisms  $\Spec(F) \rightarrow \Spec(B/\mathfrak{m}^n)$ which 
hold the targets $\Spec(B/\mathfrak{m}^n)$ fixed.}

  {To formalise this, let us write $ \SCR_{K//F}^{\art} \subset \SCR_{K//F}^{\cN}$  for the full subcategory spanned by all $B$ with $\dim_K (\pi_\ast(B)) < \infty$.  Let us fix some $B \in  \SCR_{K//F}^{\art}$ and write 
$ \theta_{B} \in  \mathrm{Var}^+_\simeq(K)^{\Delta^1}$
for  the $K$-morphism  $\Spec(F)\rightarrow \Spec(B)$;  note that 
 $\Spec(F)$ and $\Spec(B)$ indeed  belong to $\mathrm{Var}^+_\simeq(K)$.}

 {
Write $\eta_{B }\in \mathrm{Var}^+_\simeq(K)$ for the $K$-morphism $\Spec(B) \rightarrow \Spec(K)$ and 
$$\triv_{B}(A) = \Spec(A) \times_{\Spec(K)} \Spec( B ) \in   \mathrm{Var}^+_\simeq(A) \times_{ \mathrm{Var}^+_\simeq(K)}\{\eta_B\}$$  for the trivial deformation of  $\Spec( B )$ to $A$.
We note that the object $\triv_{B }(A)$ is picked out  $\ast \simeq    \mathrm{Var}^+_\simeq(K) \times_{ \mathrm{Var}^+_\simeq(K)}\{\eta_B\}
\xrightarrow{ \mathrm{Var}^+_\simeq(K \rightarrow A)\times \id }   \mathrm{Var}^+_\simeq(A) \times_{ \mathrm{Var}^+_\simeq(K)}\{\eta_B\}$, so  it \vspace{3pt} depends \mbox{functorially on $A$.}} 

 {
We  now 
consider the functor 
 $$\Def_{F/B/K}:\SCR_K^{\art} \rightarrow \mathcal{S}$$  
$$ A \mapsto\fib_{\triv_B(A)} \left(  \mathrm{Var}^+_\simeq(A)^{\Delta^1} \times_{ \mathrm{Var}^+_\simeq(K)^{\Delta^1}}\{\theta_{B }\} \xrightarrow{\mathrm{ev}_1}  
  \mathrm{Var}^+_\simeq(A) \times_{ \mathrm{Var}^+_\simeq(K)}\{\eta_{B }\}
\right).$$}

 {
Informally,  $\Def_{F/B/K }$ sends $A \in \SCR_K^{\art}$ to the space of all pullback diagrams
$$\xymatrix{\Spec(F) \ar[r] \ar[d]^{\theta_{B}} & \widetilde{Z}\ar[d]\\ \Spec(B)\ar[r] &\Spec(A) \times_{\Spec(K)} \Spec(B)}.$$
}

 {We now claim that $\Def_{F/B/K} $ satisfies the axioms of a 
 formal moduli problem over $K$.  }

 {First, we note that $\Def_{F/B/K}(K) $ is evidently a contractible space.}
 Next,    fix  a  pullback square 
		$$\xymatrix{A_3\ar[r] \ar[d] & A_2\ar[d]\\ A_1\ar[r] &A_0 }$$ in $\SCR_K^{\art}$
for which the morphisms  $\pi_0(A_1)\to\pi_0(A_0)$ and $\pi_0(A_2)\to\pi_0(A_0)$ are surjective. 
 {As $\mathrm{Var}^+_\simeq(-)$ is 
cohesive   by (the derived version of)    {\cite[Theorem 19.4.0.2]{lurie2016spectral},  applying $\mathrm{Var}^+_\simeq(-)$
to the above square gives a pullback in spaces.  Since the functor $\Def_{F/B/K}(-)$ is built from $\mathrm{Var}^+_\simeq(-)$ by operations which preserve pullbacks, we see that the following square is a pullback   in $\mathcal{S}$:
$$\xymatrix{\Def_{F/B/K}(A_3)\ar[r] \ar[d] & \Def_{F/B/K}(A_2)\ar[d]\\ \Def_{F/B/K}(A_1)\ar[r] &\Def_{F/B/K}(A_0) .}$$}
Hence $\Def_{F/B/K} $ is a formal moduli problem.}

 {The assignment $B \mapsto \Def_{F/B/K}\in \moduli_{K} \subset \Fun(\SCR_K^{\art},\mathcal{S})$ is contravariantly functorial in $B$. Indeed,  this follows from the defining pullback diagram 
$$\xymatrix{\Def_{F/B/K}\ar[r] \ar[d] & \mathrm{Var}^+_\simeq(K) \times_{ \mathrm{Var}^+_\simeq(K)}\{\eta_B\}\ar[d]^{ \mathrm{Var}^+_\simeq(K \rightarrow A)\times \id }\\ 
\mathrm{Var}^+_\simeq(A)^{\Delta^1} \times_{ \mathrm{Var}^+_\simeq(K)^{\Delta^1}}\{\theta_B\}  \ar[r]^{\mathrm{ev}_1} &\mathrm{Var}^+_\simeq(A) \times_{ \mathrm{Var}^+_\simeq(K)}\{\eta_B\} .}$$
since $\theta_B$ and $\eta_B = \mathrm{ev}_1(\theta_B)$ depend contravariantly  functorially on $B$.} 
\ \\ 
 {Hence we obtain a functor $\Def_{F/\bullet/K}: \SCR_{K//F}^{\art,\op} \rightarrow \moduli_K$}

 {We now consider the $\infty$-category $$\Pro(\SCR_{K//F}^{\art})^{\op}    \subset \Fun( \SCR_{K//F}^{\art}, \mathcal{S}) $$ of finite-limit-preserving functors $\SCR_{K//F}^{\art}\rightarrow \mathcal{S}$.  
Given $B \in \SCR_{K//F}^{\cN}$,  the functor $$\Map_{\SCR_{K//F}}(B,-): \SCR_{K//F}^{\art}\rightarrow \mathcal{S}$$
belongs to $\Pro(\SCR_{K//F}^{\art})^{\op}$, and this  assignment gives a functor $Y:  (\SCR_{K//F}^{\cN})^{\op} \rightarrow  \Pro(\SCR_{K//F}^{\art})^{\op} $.}

 { 
Since  $\Pro(\SCR_{K//F}^{\art})^{\op}  =  \Ind( \SCR_{K//F}^{\art,\op}) $,  we can use the universal property of the $\Ind$-construction (cf.\ \cite[Proposition 5.3.5.10]{lurie2009higher}) to extend the functor $\Def_{F/\bullet/K}: \SCR_{K//F}^{\art,\op} \rightarrow \moduli_K$
  in a filtered-colimit-preserving way to a functor  $\Pro(\SCR_{K//F}^{\art})^{\op}  \rightarrow \moduli_K$. Precomposing with the Yoneda functor $Y:  (\SCR_{K//F}^{\cN})^{\op}  \rightarrow  \Pro(\SCR_{K//F}^{\art})^{\op}$ gives an extension 
$ (\SCR_{K//F}^{\cN})^{\op}   \rightarrow \moduli_K  $
of $\Def_{F/\bullet/K} $  from  $\SCR_{K//F}^{\art,\op}$ to   $(\SCR_{K//F}^{\cN})^{\op}$.  As this functor  sends  $K$ to $\Def_{F/K}$, we obtain a lift
$$ \Def_{F/\bullet/K} : (\SCR_{K//F}^{\cN})^{\op} \rightarrow (\moduli_{K})_{/\Def_{F/K}}.$$Note that we have slightly abused notation by also using the name  $ \Def_{F/\bullet/K}$ for this \vspace{5pt}  new functor.}

  {We will now show that the tangent fibre of the formal moduli problem $ X:=\Def_{F/B/K}$ is  $L_{F/B}^\vee[1]$.}
Tangent fibres of   formal moduli problems   deforming   morphisms under constraints
are well-known to experts.
We  outline the main steps of the computation   for the reader's convenience, and refer to \cite[Proposition 6.4.19]{nuiten2018lie} or \vspace{1pt}  \cite[Proposition 3.11]{porta2020non} for further details.

First,  let us assume that $B \in \SCR_{K//F}^{\art}$ is Artinian.  Recall from \eqref{functorialtangent} that the tangent fibre $T_X\in \Mod_K$ is characterised by a natural equivalence 
$$  \   \   \   \ \Map_K(V^\ast, T_X) \simeq  X( K\oplus V)\  ,    \   \ V\in \Perf_{K, \geq 0}$$
 where $(-)^\ast$ denotes $K$-linear duality. \vspace{3pt}In what follows below,  will write  $(-)^\vee$ for \mbox{$F$-linear duality.}
By \Cref{finite} and the fundamental cofibre sequence,  
$L_{F/B}$ is of finite type.  We therefore obtain a chain of natural \vspace{3pt} equivalences  for any
 $V\in \Perf_{K, \geq 0}$:
 {\begin{align*}
\Map_K( V^\ast, L_{F/B}^\vee[1]) & \simeq \Map_F(( F\otimes_K V)^{\vee} , L_{F/B}^\vee[1]))& \\
 & \simeq \Omega^\infty \left( ( F \otimes_K V) \otimes_F L_{F/B}^\vee[1]\right)& \\
	&\simeq \Map_F(L_{F/B}, F\otimes_K V[1] )&
	\end{align*}}
 To show that $T_{X} \simeq L_{F/B}^\vee[1]$, it therefore suffices to construct  a 
 \mbox{natural equivalence of spaces}
$$\Map_F(L_{F/B}, F\otimes_K V[1] ) \simeq \Def_{F/B/K}(K\oplus V) \  ,    \   \ V\in \Perf_{K, \geq 0}.\vspace{1pt} $$
 {Indeed,  using the universal property of the cotangent complex and the fact that $B\rightarrow F$ is a map over $K$,   we can identify the space of 
 $F$-linear maps \mbox{$L_{F/B} \rightarrow F \otimes_K V[1]$} with the fibre of the map
$$ \Map_{\SCR_K} \left(F, F \oplus (F \otimes_K V[1]) \right) \rightarrow \Map_{\SCR_K}  \left(B, F \oplus (F \otimes_K V[1])\right) \times_{\Map_{\SCR_K}  \left(B, F\right) } \Map_{\SCR_K}  \left(F, F\right) $$
over the point $$(B \xrightarrow{(\epsilon,0)} F \oplus (F \otimes_K V[1]),  \  \id_F  );$$ here $\epsilon: B \rightarrow F$ is the \vspace{3pt} structure \mbox{morphism of $B$.}}

  {Equivalently,   the space of 
 $F$-linear maps \mbox{$L_{F/B} \rightarrow F \otimes_K V[1]$}  can be identified with the  
space of maps    \mbox{$ \alpha:\Spec(F\oplus(F\otimes_K V[1]))\to\Spec(F)$ } rendering commutative the following diagram: }

$$\xymatrix{
\Spec(F) \ar[r] \ar[d] & \Spec(F \oplus (F \otimes_K V[1]) )  \ar[d]  \ar@{-->}[r]^{ \ \ \ \ \ \ \ \ \ \  \alpha}&\Spec(F)\ar[d]\\
\Spec(B) \ar[r] \ar[d] & \Spec(B \oplus (B \otimes_K V[1]) )  \ar[r] \ar[d] &  \Spec(B)\ar[d]\\ 
\Spec(K)  \ar[r] & \Spec(K \oplus V[1] )  \ar[r]   &  \Spec(K).}\vspace{-2pt}$$\\
Here horizontal composites are identity maps, and the middle and lower horizontal maps on the right correspond  to 
$(\id_B, 0)$
 and 
$(\id_K, 0)$, respectively. \vspace{1pt}
 
 {The top right square is a homotopy pullback,  and we therefore obtain two equivalences 
$$  \Spec(B \oplus (B \otimes_K V[1]) ) \hspace{-4pt}  \mytimes{ \Spec(B)}  \hspace{-4pt}\Spec(F)   \xleftarrow{{(\id_F,0)}}  \Spec(F \oplus (F \otimes_K V[1]) )  \xrightarrow{\alpha}  \Spec(B \oplus (B \otimes_K V[1]) ) \hspace{-4pt}  \mytimes{ \Spec(B)} \hspace{-4pt} \Spec(F) $$}
from which we obtain  a point in the space
$X(K) \times_{X(K\oplus V[1])} X(K) \simeq \Omega(X(K \oplus V[1]))\simeq X(K \oplus V),  $
 of   {auto}morphisms of the trivial deformation.\\
 {The resulting map 
$ \Map_F(L_{F/B}, F\otimes_K V[1] ) \rightarrow \Omega(X(K \oplus V[1])) \simeq X(K \oplus V[1])$ is an equivalence,   
and  we deduce that  $T_{\Def_{F/B/K}} \simeq L_{F/B}^\vee[1]$. A diagram chase now
shows that the map}  \vspace{-1pt}
$$\Def_{F/B/K} \rightarrow \Def_{F/K/K}\simeq \Def_{F/K}\vspace{2pt}$$ induces the natural morphism $L_{F/B}^\vee[1] \rightarrow L_{F/K}^\vee[1]$ on tangent fibres for all $B \in \SCR_{K//F}^{\art}$ Artinian.
But this in fact holds true for all $B\in \SCR_{K//F}^{\cN}$, because 
 \Cref{AffineKD} implies  that 
 the functor $(\SCR_{K//F}^{\cN})^{\op}\rightarrow (\Mod_{F, \leq 0}^{\ft})_{/ L_{F/K}^\vee [1]}$, $B \mapsto (L_{F/B}^\vee [1] \rightarrow L_{F/K}^\vee [1])$ preserves filtered colimits.\vspace{5pt}

Returning to the  statement of the theorem,  we will now construct the forgetful functor  \vspace{-3pt}
$$U: \Alg_{\LieAlgd^\pi_{F/K}} \rightarrow (\Alg_{\Lie^\pi_{K}})_{/L_{F/K}^\vee[1]}. \vspace{-3pt}$$
As in\vspace{-2pt}  \Cref{fibre_anchor}, let 
$\mathcal{D}_0$ contain  all 
   $(\mathfrak{g} \xrightarrow{\rho} L_{F/K}^\vee [1])\in \Alg_{\LieAlgd^\pi_{F/K}} $ with \mbox{$\fib(\rho) \in \Mod_{F, \leq 0}^{\ft}$.} Let $\mathcal{C}_0 \subset \mathcal{D}_0$ consist of all\vspace{2pt} free algebroids on objects $(V \xrightarrow{0} L_{F/K}[1])$ with $V \in \Perf_{F,\leq 0}$.
Recall that if 
$(\mathfrak{g} \xrightarrow{\rho} L_{F/K}^\vee [1])\in \mathcal{D}_0$  corresponds to $B \in \SCR_{K//F}^{\cN} $
under the equivalence in \Cref{AffineKD}, then 
 the underlying object of $(\mathfrak{g} \xrightarrow{\rho} L_{F/K}^\vee [1])$ is   $(L_{F/B}^\vee [1] \xrightarrow{\rho} L_{F/K}^\vee [1])$. 
We  form the\vspace{-2pt} \mbox{composite functor $G_0$:}
		$$ \mathcal{C}_0 \subset \mathcal{D}_0\simeq(\SCR_{K//F}^{\cN})^\mathrm{op} \xrightarrow{\Def_{F/\bullet/K}} (\mathrm{Moduli}_{K})_{/\Def_{F/K}}\xrightarrow{\mathrm{[BM19, 1.11]}}(\Alg_{\Lie_{K}})_{/L_{F/K}^\vee[1]}
		 $$
On the other hand, the\vspace{-2pt} \vspace{-2pt}  composite $$ \mathcal{D}_0\simeq(\SCR_{K//F}^{\cN})^\mathrm{op} \xrightarrow{} (\Alg_{\Lie_{K}})_{/L_{F/K}^\vee[1]}\xrightarrow{\Forget} (\Mod_F)_{/L_{F/K}^\vee[1]}\vspace{-2pt} 
$$ sends $B\in \mathcal{D}_0 $ to the object  $(L_{F/B}^\vee [1] \xrightarrow{\rho} L_{F/K}^\vee [1])\in (\Mod_F)_{/L_{F/K}^\vee[1]}$, by our above computation of the map of tangent complexes $T_{\Def_{F/B/K}} \rightarrow T_{\Def_{F/K}}$.
Hence $\Forget \circ G_0$ naturally extends to 
a sifted-colimit-preserving functor   $ {H} : \Alg_{\LieAlgd^\pi_{F/K}} \rightarrow (\Mod_K)_{/L_{F/K}^\vee[1]}$
and   \Cref{functor_extensions} provides  the desired sifted-colimit-preserving extension $G$ of $G_0$. 
\end{proof}
\newpage

 \newpage
\section{The Fundamental Theorem}\label{sec:fundamental_theorem}

To establish a  Galois correspondence for finite  purely inseparable field extensions $F/K$, we will use partition Lie algebroids (cf.\ \Cref{palgebroiddef}) as a natural substitute for  the restricted Lie algebroids appearing in 
Jacobson's exponent one\vspace{3pt} correspondence in \Cref{Jacob}.
 
In \Cref{AffineKD}, we have seen  that the natural\vspace{-2pt} tangent fibre functor
$$\mathfrak{D}: \SCR_{K//F} ^{\op} \rightarrow \Alg_{\LieAlgd^\pi_{F/K}} $$
lifting the assignment $ B \mapsto (L_{F/B}^\vee [1] \rightarrow L_{F/K}^\vee [1])$ 
becomes fully faithful after restriction to $(\SCR_{K//F}^{\cN}) ^{\op}$, the subcategory of complete local Noetherian objects\vspace{3pt} in $(\SCR_{K//F}) ^{\op}$.

As intermediate fields $K\subset E \subset F$ are in particular objects of  $\SCR_{K//F}^{\cN}$, we obtain a description of intermediate fields in terms of partition Lie algebroids.
To  complete the proof of  \Cref{fundamental_theorem}, it therefore suffices to characterise the essential image of $\mathfrak{D}|_{\fields_{K//F}}$, the 
 restriction of $\mathfrak{D}$ to the full subcategory
$\fields_{K//F} \subset  \SCR_{K//F}$  spanned by  intermediate fields. Note that as homomorphisms between fields are injective, 
$\fields_{K//F}$   in fact forms a poset.

\subsection{The homotopical algebra of field extensions}
Before characterising this essential image, we will review several  elementary facts concerning the   \vspace{3pt}  cotangent \mbox{complex of field extensions $F/K$.}

We start by recalling a \vspace{-3pt}  well-known and fundamental computational tool, \mbox{cf.\ e.g.\ \cite[\href{https://stacks.math.columbia.edu/tag/08SJ}{Tag 08SJ}]{stacks-project}:}
\begin{lemma}\label{regsurj}
If $f:A \rightarrow B$ is a  surjective map of commutative rings whose kernel $I$ is generated by a regular sequence, then the relative cotangent complex is given by $$L_{B/A} \simeq I/I^2[1].$$
\end{lemma} 

With this lemma, we can  easily compute  the  relative cotangent complex for  finite field extensions, as these are complete intersections. We obtain the following    classical result:

\begin{proposition}\label{prop:cotangent_calculation}
	Let $F/K$ be a finite field extension.  Pick $x_1,\ldots, x_n \in F$ such that the map $\phi:K[X_1, \ldots, X_n]\xrightarrow{X_i\mapsto x_i} F$ is surjective, and write $I$ for  the kernel of $\phi$.  There is an equivalence 
	$$ L_{F/K}\simeq \left( \ldots \rightarrow 0 \rightarrow  I/I^2 \rightarrow \Omega^1_{K[X_1,\ldots, X_n]/K} \otimes_{K[X_1,\ldots, X_n]}  F\right), $$
	where the boundary map sends a class $[i] \in I/I^2$ to the element $di \otimes 1$.
\end{proposition} 
\begin{proof}
	For $k\geq 0$, we write $K[x_1,\ldots,x_k]\subset F$ for the subring generated by $x_1,\ldots, x_k$. Since $K[x_1,\ldots,x_k]$ is a finite domain over $K$, it is in fact a field.
	We can inductively pick polynomials   $$P_1,\ldots, P_n \in K[X_1,\ldots, X_n] $$
	such that $P_i  $ belongs to $ R=K[X_1,\ldots, X_i]$,  is monic in $X_i$ (over $K[X_1,\ldots, X_{i-1}]$), and maps to the minimal polynomial of $x_i$ over $K[x_1,\ldots,x_{i-1}]$.
As the polynomials are monic,  $P_1,\ldots, P_n$ form a regular sequence generating  $I = \ker(K[X_1,\ldots, X_n] \xrightarrow{X_i \mapsto x_i} F)$.
	Combining \Cref{regsurj} \mbox{with   the cofibre sequence} $$F\otimes_R L_{R/K}\to L_{F/K}\to L_{F/R}, $$
	we see that $L_{F/K}$ is the cofibre of a map $I/I^2\to  \Omega^1_{K[X_1,...,X_n]/K}  \otimes_{K[X_1,...,X_n]} F$.  
The boundary map can be identified with  $[i]\mapsto di\otimes 1$ using the classical conormal sequence, see for example \cite[Proposition 16.3]{eisenbud}. 
\end{proof}

Hence for any finite field extension $F/K$, the homology of $L_{F/K}$ is concentrated in two degrees. 

The difference of the
nonzero homology groups measures how far $F/K$ is from being algebraic.
The following result of Cartier, which appears in \cite[Th\'{e}or\`{e}me 0.21.7.1]{EGA} or    \cite[\href{https://stacks.math.columbia.edu/tag/07E1}{Tag 07E1}]{stacks-project}, 
will play an important role in our main argument:
 
\begin{lemma}[Cartier's equality] \label{Cartier}
Let $F/K$ be a finitely generated field extension. Then the module of K\"{a}hler differentials 
$\Omega^1_{F/K} = \pi_0(L_{F/K})$ and the module of imperfection $\Upsilon_{F/K} = \pi_1(L_{F/K})$ are finite-dimensional, and satisfy 
$$\dim_F( \Omega^1_{F/K}) - \dim_F(\Upsilon_{F/K} ) = \trdeg_K(F), $$
where the right hand side denotes the transcendence degree of $F$ over $K$.
\end{lemma}

 {Note that in view of \Cref{prop:cotangent_calculation}, for a finite extension, the integer which appears in Cartier's equality is actually given by the Euler characteristic $\chi(L_{F/K})$.}

\begin{remark}
The \textit{module of imperfection} was originally defined as \mbox{$\Upsilon_{F/K} = \ker(\Omega^1_{K} \otimes_K F \rightarrow \Omega^1_F)$,}
 cf.\ e.g.\ \cite[Th\'{e}or\`{e}me 0.21.7.1]{EGA}, where $\Omega^1_{R} = \Omega^1_{R/\ZZ} $ is the module of absolute K\"{a}hler differentials. 
It is well-known that $\Upsilon_{F/K} \cong \pi_1(L_{F/K})$  for fields  (cf.\   \cite[Lemma 1.1.2]{saito2020graded} for a modern reference), and that $\Upsilon_{F/K}$ vanishes precisely if $F/K$ is separable, cf.\   \cite[Proposition 0.20.6.3]{EGA}.
\end{remark}

\subsection{The essential image theorem}
Finally, we come to the main result of this section,
in which we characterise   the essential image of the functor from intermediate field extensions $K \subset E \subset F$ to $F/K$-partition Lie algebroids.  This will allow us to complete  the proof of the fundamental theorem of purely inseparable Galois theory,   \Cref{fundamental_theorem}.

\begin{theorem}[Essential Image Theorem]\label{essentialimagetheorem}
	Let $F/K$ be a finite purely inseparable field extension.  An $F/K$-partition Lie algebroid
	$(\mathfrak{g}\xrightarrow{\rho}L_{F/K}^\vee[1])\in \Alg_{\LieAlgd^\pi_{F/K}}$ is equivalent to one of the form $$\mathfrak{D}(E) = (L_{F/E}^\vee[1]\xrightarrow{\rho}L_{F/K}^\vee[1])\vspace{3pt}$$  for some intermediate field $K\subset E \subset F$ if and only if the following conditions are satisfied:
	
	\begin{enumerate}
		\item Injectivity: the anchor map $\rho$ induces an injection $\pi_1(\mathfrak{g}) \hookrightarrow \pi_1(L_{F/K}^\vee[1]) \cong \Der_K(F)$.
		\item  Vanishing: $\pi_k(\mathfrak{g})=0$ for $k\neq 0,1$.
		\item Balance: $\dim_F(\pi_0(\mathfrak{g})) = \dim_{F}(\pi_{1}(\mathfrak{g}))<\infty$.
		\vspace{5pt}
	\end{enumerate}

\end{theorem}

\begin{proof}
	First, suppose that $(\mathfrak{g}\xrightarrow{\rho}L_{F/K}^\vee[1])$ is equivalent to $\mathfrak{D}(E)  = (L_{F/E}^\vee[1]\xrightarrow{\rho}L_{F/K}^\vee[1])$ for some intermediate field $E$.  Condition (1)  then follows as we can identify the map $\pi_1(\mathfrak{g}) \hookrightarrow \pi_1(L_{F/K}^\vee[1])$ with Jacobson's inclusion $\Der_E(F)\hookrightarrow\Der_K(F)$ of $E$-linear derivations into $K$-linear derivations.  Condition (2) follows from the computation of the cotangent complex of   finite  field extensions in   \Cref{prop:cotangent_calculation}.  Finally, condition (3) follows from  Cartier's equality in \Cref{Cartier}, since $F/E$ is an algebraic extension.\vspace{3pt}
	
	For the converse implication, we assume that \vspace{-2pt} $(\mathfrak{g}\xrightarrow{\rho}L_{F/K}^\vee[1])\in \Alg_{\LieAlgd^\pi_{F/K}} $ satisfies $(1)-(3)$.  First, we  show that there is an  $R\in \SCR_{K//F}^{\cN}$ with $\mathfrak{D}(R)\simeq (\mathfrak{g}\xrightarrow{\rho}L_{F/K}^\vee[1])$
	as \mbox{partition Lie algebroids.}
	Indeed, by \Cref{AffineKD}, it is enough to check that $\mathrm{fib}(\rho)$ is coconnective and of finite type. 
	\mbox{Conditions} (2) and (3) together imply that 
	$\mathfrak{g}$ is of finite type,  and the same holds true for $L_{F/K}^\vee[1]$. As
	$\fib(\rho)$ fits into a cofibre sequence with $\mathfrak{g}$ and $L_{F/K}^\vee[1]$, it is of finite type as well.  
	
	To show that $\fib(\rho)$ is coconnective, we look at the exact sequences
	$$ \pi_{n+1}(L_{F/K}^\vee[1])\to\pi_{n}(\fib(\rho))\to \pi_n(\mathfrak{g})\to \pi_n(L_{F/K}^\vee[1]).$$
	For $n\geq 2$, the group $\pi_n(\mathrm{fib}(\rho))$ is nested between two vanishing modules, so must itself be zero.  
	For $n=1$,  $\pi_1(\mathrm{fib}(\rho))$ vanishes since  $\pi_{2}(L_{F/K}^\vee[1])=0$ and   $\pi_1(\mathfrak{g})\hookrightarrow \pi_1(L_{F/K}^\vee[1])$ is injective by  (1).
	
	We may therefore assume  that  $(\mathfrak{g}\xrightarrow{\rho}L_{F/K}^\vee[1])$ is 
	equivalent to  $\mathfrak{D}(R) = (L_{F/R}^\vee[1]\xrightarrow{\rho} L_{F/K}^\vee[1])$  for some complete local Noetherian object $R\in \SCR_{K//F}^{\cN}$. It remains to prove that $R$ is in fact a field, i.e. a discrete simplicial commutative ring which is 
	regular, local, and of dimension zero.\vspace{2pt} \vspace{3pt} 
	
	Observe that since $F$ is a field,  we may reformulate conditions (2) and (3) as 
	\begin{enumerate}
		\item[(2')] Vanishing:  $\pi_k(L_{F/R})=0$ for $k\neq 0, 1$.
		\item[(3')] Balance:  $\dim_F\pi_0 (L_{F/R})=\dim_F\pi_1(L_{F/R})<\infty$.\vspace{3pt} 
	\end{enumerate}
	
	Denote the residue field of $\pi_0(R)$ by $E$.  We first reduce to the case of $F=E$ by showing that the three conditions for $L_{F/R}$ imply the three conditions for $L_{E/R}$.
	First consider the sequence 
	\[L_{E/R}\otimes_EF\to L_{F/R}\to L_{F/E}\]
	It follows that $L_{E/R}$ has non-zero homology only in degrees $0$ and $1$ since it is connective, and the other two terms of the sequence also vanish outside these degrees. For $L_{F/R}$ this is by assumption while for $L_{F/E}$ it is because this is a finite purely inseparable extension of fields.  This verifies the vanishing condition.  Furthermore, the balance condition follows from \[0=\chi(L_{F/R})=\chi(L_{E/R})+\chi(L_{F/E})=\chi(L_{E/R}).\]  
	To check that the injectivity condition $\pi_1(L_{E/R}^\vee[1])\to \pi_1(L_{E/K}^\vee[1])$ holds true,  it is enough to verify that {$\pi_2((L_{R/K}\otimes_R E)^\vee[1])=0$, } but this follows from the corresponding fact about $K\to R\to F$ and the fact that $E\to F$ is faithfully flat.

		 {Now we show that $\pi_0(R)$ is regular.	}
	The  maps $R\rightarrow \pi_0(R) \rightarrow E$ induce an exact sequence
	$$  \pi_2( L_{E/R}) \rightarrow \pi_2(L_{E/\pi_0(R)})\rightarrow \pi_1(E\otimes_{\pi_0(R)} L_{\pi_0(R)/R} ).$$
	As the fibre of $R\rightarrow \pi_0(R)$ is $1$-connective, $L_{\pi_0(R)/R}$ is $2$-connective by  \cite[Corollary 25.3.6.4]{lurie2016spectral}, which implies that $\pi_1(E \otimes_{\pi_0(R)} L_{\pi_0(R)/R} ) = 0$. Since we have also proven that $\pi_2(L_{E/R})=0$, we  deduce that $ \pi_2(L_{E/\pi_0(R)})=0$.
	By \cite[Corollary 10.5]{Quillen1},
	we conclude that $\pi_0(R)$ is regular.  Note that \Cref{regsurj} implies that $L_{E / \pi_0(R)} \simeq  (\mathfrak{m}/ \mathfrak{m}^2)[1]$, where $ \mathfrak{m}$ is the maximal ideal of $\pi_0(R)$. In particular, we see that  $\pi_i(L_{E / \pi_0(R)})$ vanishes for all  $i\neq 1$.

	 {We next show that $R$ is discrete, for which it suffices by \cite[Corollary 25.3.6.6]{lurie2016spectral}
		to  show that the relative cotangent complex $L_{\pi_0(R)/R}$ \mbox{vanishes.}}
	{We have already} seen that $L_{\pi_0(R)/R}$ is  $2$-connective, i.e. that 
	$\pi_n(L_{\pi_0(R)/R}) = 0$ for  {$n\leq 1$}.  So fix $n\geq 2$, and  assume we have already checked  {$\pi_k(L_{\pi_0(R)/R}) = 0$ for all $k<n$.}
	The maps $R \rightarrow \pi_0(R) \rightarrow E$ give  rise to an exact sequence
	$$ \pi_{n+1}(L_{E/\pi_0(R)}) \rightarrow \pi_n(E \otimes_{\pi_0(R)} L_{\pi_0(R)/R}) \rightarrow \pi_n(L_{E/R}).$$
	We have already seen that the two terms on the outside vanish for $n>1$ and therefore we deduce that 
	$\pi_n(E \otimes_{\pi_0(R)} L_{\pi_0(R)/R}) \cong E \otimes_{\pi_0(R)} \pi_n(L_{\pi_0(R)/R})= 0$.
	 {The isomorphism here comes from the fact $n$ is the lowest non-vanishing homology, so this $\mathrm{Tor}$ group is just a tensor product.  Nakayama's lemma then implies that $\pi_n(L_{\pi_0(R)/R})= 0$. To show that Nakayama applies, note that }
	$\pi_0(R)$ is almost of finite presentation over $R$ by \cite[Proposition 3.1.5]{lurie2004derived} since $R$ is Noetherian, which implies by Proposition 3.2.14 in [op.cit.]  that $L_{\pi_0(R)/R}$ is an almost perfect $\pi_0(R)$-module spectrum, which in turn  shows that $\pi_n(L_{\pi_0(R)/R})$ is finitely generated by  Proposition 2.5.10 in [op.cit.].

We deduce that $R \simeq \pi_0(R)$ is a discrete  regular local  ring with maximal ideal $\mathfrak{m}$.
On the other hand since $\pi_0(R)\to E$ is surjective,  we know that $\pi_0(L_{E/R})=0$ and hence,  by balance,  we see that $\pi_1(L_{E/R})=0$.   Hence $\mathfrak{m}/ \mathfrak{m}^2 \cong \pi_1( L_{E / R} ) = 0$ and so $R$ is a discrete  regular local  ring of dimension zero, i.e.\ a field.
\end{proof}

We can now prove the main theorem:\vspace{-2pt}
\begin{proof}[Proof of \Cref{fundamental_theorem}]
By \Cref{AffineKD}, the adjunction $\mathfrak{D}: \SCR_{K//F} ^{\op}\leftrightarrows \Alg_{\LieAlgd^\pi_{F/K}}: C^\ast $ from \Cref{constructadjunction} restricts to an equivalence between $(\SCR_{K//F}^{\cN}) ^{\op}$
and the full subcategory of $F/K$-partition Lie algebroids for which the fibre of the anchor map belongs to $\Mod_{F,\leq 0}^{\ft}$.
The Essential Image \Cref{essentialimagetheorem} shows that further restricting $(\mathfrak{D}\dashv  C^\ast)$ to the subcategory $\fields_{K//F} ^{\op}$ of intermediate fields therefore gives an equivalence between $\fields_{K//F} ^{\op}$ and the full subcategory of $F/K$-partition Lie algebroids satisfying the  conditions $(1)-(3)$ appearing in the theorem. In particular, this subcategory of $F/K$- partition Lie algebroids\vspace{-3pt} is  \mbox{(equivalent to) a poset.}
\end{proof} 

\begin{proof}[Proof of \Cref{geometric_cor}]
 {
Let $X$ be a normal variety over a perfect field $k$ with fraction field $F$.  There is an equivalence of categories between intermediate fields $F/K/F^{p^n}$ and towers of finite $k$-morphisms $X\to Y\to X^{p^n}$ with $Y$ normal.  The forward direction is given by taking the normalization of $\mathcal{O}_{X^{p^n}}$ in $K$, while the reverse is given by taking the fraction field.   It follows that the correspondence with normal varieties follows immediately from the correspondence with field extensions.
}
\end{proof}
 
 \subsection{Modular extensions}
 We will now characterise simple and modular extensions in terms of their associated partition Lie algebroids.   {We begin with some elementary lemmas.}

 \begin{lemma}\label{lem:simple_2}
 { Let $F/K$ be a finite purely inseparable field extension, and let $\alpha\in F$.  Then $F=K(\alpha)$ if and only if $F=(F^pK)(\alpha)$.}
 \end{lemma}
 \begin{proof}
  {If $F=K(\alpha)$, then $F$ is also generated by $\alpha$ over  $F^pK$ since $K\subset F^pK$.  }
	
	Conversely,  suppose that $F=F^pK(\alpha) $, and consider the tower of extensions
	\[F\supseteq F^pK\supseteq F^{p^2}K\supseteq\dots \supseteq F^{p^n}K=K.\]
	 {We show inductively on $e\in\mathbb{N}_{>0}$ that the extension $F/F^{p^e}K$ is generated by $\alpha$.  This holds for $e=1$ by assumption.  Now assume that the statement holds for $e$.  
Since taking $p^{th}$-powers commutes with addition,  we have  $F^p=F^{p^{e+1}}K^p(\alpha^p)$ and hence
	$F^pK=F^{p^{e+1}}K(\alpha^p).$ We can therefore conclude 
	$F=F^{p}K(\alpha)
	=(F^{p^{e+1}}K(\alpha^p))(\alpha)=F^{p^{e+1}}K(\alpha)$.  Thus the statement holds for $e+1$.  The result follows since the tower of extensions above must terminate since $F/K$ is finite.}

 \end{proof}
 
  {Recall that given a finite purely inseparable extension $F/K$, a subset $\{x_i\}\subset F$ is a \emph{$p$-basis} of $F/K$ if $\{dx_i\}$ form a basis of $\Omega_{F/K}$.  By \cite[07P2]{stacks-project}, this is equivalent to the elements $\{\Pi_i x_i^{k_i}\mid 0\leq k_i\leq p-1\}$ forming a basis of $F$ over $KF^p$.  The latter is often given as the definition of $p$-basis, particularly when $F/K$ has exponent one.}

\begin{lemma}\label{simple}
	Let $F/K$ be a finite purely inseparable field extension.  Then $F/K$ is a non-trivial simple extension if and only if $\dim_F\pi_0(L_{F/K})=\dim_F\pi_1(L_{F/K})=1$.
	\end{lemma}
\begin{proof}
By Lemma \ref{Cartier}, it is sufficient to show that $F/K$ is a non-trivial simple extension if and only if $\dim_F\Omega^1_{F/K}=1$.
	
	If $F/K$ has exponent $1$, then $F/K$ is a non-trivial simple extension if and only if it has a $p$-basis consisting of a single element by \cite[07P2]{stacks-project}.
	
 {	If $F/K$ has exponent $n>1$, we note that since $\Omega^1_{F/K}\cong\Omega^1_{F/F^pK}$,  it suffices to show that $F/K$ is simple if and only if the exponent $1$ extension $F/F^pK$ is simple by \Cref{lem:simple_2}, and $F/K$ is non-trivial if and only if $F/F^pK$ is simple by \Cref{lem:simple_2} applied to $\alpha=1$.}
	\end{proof}
	
	\begin{remark}
	We note that Lemma \ref{simple} is consistent with our main theorem, despite the fact that it shows that if we have a chain of non-trivial finite simple extensions \[K=E_n\subset E_{n-1}\subset ...\subset E_1\subset E_0=F\] where there is a single element $\alpha\in F$ such that $E_{i-1}=E_{i}(\alpha^{p^i})$, then the homotopy groups of the corresponding partition Lie algebroids $\mathfrak{gal}_{F/K}(E_i)$ are all isomorphic as	 $F$-modules.
	
Indeed, consider the following exact sequence consists of one dimensional $F$-modules:	
			\[ 0\to \pi_1(L_{F/E_{i-1}}^\vee[1])\to\pi_1(L_{F/E_{i}}^\vee[1])\to\pi_1(F\otimes_{E_{i-1}} L_{E_{i-1}/E_i}^\vee[1])\] \[\to \pi_0(L_{F/E_{i-1}}^\vee[1])\to\pi_0(L_{F/E_{i}}^\vee[1])\to\pi_0(F\otimes_{E_{i-1}} L_{E_{i-1}/E_i}^\vee[1])\to 0	\]
	By the one-dimensionality, the first  injective map is  an isomorphism.  Hence following through the sequence we find that $\pi_0(L_{F/E_{i-1}}^\vee[1])\to\pi_0(L_{F/E_{i}}^\vee[1])$ is the zero map, and hence there is no isomorphism of partition Lie algebroids  {even though the underlying $F$-modules are isomorphic}.  A similar argument applies to a sequence of iterated Frobenius maps \[K=F^{p^n}\subset F^{p^{n-1}}\subset...\subset F^p\subset F,\] for an $F$-finite field $F$, where the homotopy groups have dimension $\dim_F(\Omega^1_{F/F^p})$.
			\end{remark}
 
 Now we reach the characterisation of modular extensions:

\begin{proposition}
	Let $F/K$ be a finite purely inseparable extension. 
	Then $F/K$ is modular precisely if there  are finitely many $F/K$-partition Lie algebroids 
	$$\rho_i:\mathfrak{g}_i\to L_{F/K}^\vee[1]$$
	such that the following conditions hold:
	\begin{enumerate}
		\item each $ \mathfrak{g}_i $ satisfies conditions $(1)-(3)$ of Theorem \ref{fundamental_theorem};
		\item  $\dim_F(\pi_0(\fib(\rho_i)))=1$ for each $i$;
		\item the canonical map $L_{F/K}^\vee\rightarrow \oplus_i\fib(\rho_i)$ is an equivalence in $\Mod_F$.
	\end{enumerate} 
\end{proposition}

\begin{proof}
	Suppose that $F/K$ is modular.  Then by definition it can be expressed as a tensor product of simple extensions $F\cong \otimes_K E_i$, where $E_i=K(\alpha_i)\subset F$.  Setting $ \mathfrak{gal}_{F/K}(E_i) = (\mathfrak{g}_i \xrightarrow{\rho_i} L_{F/K}^\vee[1])$, we note that $(1)$ holds by  \Cref{fundamental_theorem} and $(2)$ holds by \Cref{simple}. For $(3)$, we use \cite[09DA]{stacks-project}  to conclude that $L_{F/K}\cong \oplus_i (F\otimes_{E_i} L_{E_i/K})$, as required.\vspace{3pt} 
	
	Conversely suppose that we have  $$\rho_i:\mathfrak{g}_i\to L_{F/K}^\vee[1]$$ satisfying the given conditions.  Then by Theorem \ref{fundamental_theorem} there are intermediate fields $K\subset E_i\subset F$ such that $\mathfrak{gal}_{F/K}(E_i) = (\mathfrak{g}_i \xrightarrow{\rho_i} L_{F/K}^\vee[1])$, and $E_i/K$ are simple by Lemma \ref{simple}, again using the fact that $\fib(\rho_i)\cong (F\otimes_{E_i} L_{E_i/K})^\vee$.  Let $R=\otimes_K E_i$, and let $f:R\to F$ be the natural map, which we must show is an isomorphism.  Again by \cite[09DA]{stacks-project}, we have $$L_{R/K}\simeq \oplus_i (R\otimes_{E_i} L_{E_i/K}).$$
	The fibre sequence associated with  $K\to R\to F$ is  $F\otimes_R L_{R/K}\to L_{F/K}\to L_{F/R}$, and by assumption, the left hand map is an equivalence, and hence $L_{F/R}\simeq 0$.

	 {
	Note that $R\to F$ is of finite presentation since $F$ is a finite extension of the residue field of $R$, so it follows that $R\to F$ is \'etale.  Now by \cite[025G]{stacks-project}, $\Spec(F)\to\Spec(R)$ is an open immersion, and therefore an isomorphism since $F$ is a field.}
\end{proof}

\newcommand{\Fbox}[1]{\fbox{\strut#1}}
\setlength{\fboxsep}{1pt} 
 
\tableofcontents

\newpage

\newpage
\bibliographystyle{amsalpha}
\bibliography{library}

\end{document}